\documentclass[12pt,a4paper]{article}
\usepackage[a4paper,top=2.5cm,bottom=2.5cm,left=2.5cm,right=2.5cm]{geometry}
\usepackage{amsmath}
 \usepackage{graphicx}
\usepackage{amsthm}
\usepackage{color}
\usepackage{fancyhdr}
\usepackage{pdfpages}
\usepackage{amssymb}
\usepackage{paralist}
\usepackage{extarrows}
\usepackage{setspace}
\usepackage{braket} 
\usepackage[colorlinks=true,citecolor=blue,linkcolor=blue,anchorcolor=blue,urlcolor=blue]{hyperref}
\allowdisplaybreaks
\numberwithin{equation}{section}
\newtheorem{theorem}{Theorem}[section]
\newtheorem{lemma}[theorem]{Lemma}
\newtheorem{proposition}[theorem]{Proposition}
\newtheorem{corollary}[theorem]{Corollary}

\begin{document}
\title{The  energy identity of Sacks-Uhlenbeck operator and infinitely many solutions for  Brezis-Nirenberg problem}
\author{Fei Fang\footnote{E-mail addresses: fangfei68@163.com}  \\  \footnotesize  \emph{Department of Mathematics,Beijing Technology and Business University, Beijing 100048, China}}

\maketitle
\noindent \textbf{\textbf{Abstract:}} Let $\Omega$ be a bounded smooth  domain in $\mathbb{R}^N$ with $N\geq 3$, $1<\alpha$, $2^{\ast}=\frac{2N}{N-2}$ and
$\{u_\alpha\}\subset H_{0}^{1,2\alpha}(\Omega)$   be a critical point of the functional
\begin{equation*}
I_{\alpha,\lambda}(u)=\frac{1}{2\alpha}\int\limits_{\Omega} [(1+|\nabla u|^2)^{\alpha}-1 ]dx-\frac{\lambda}{2}\int\limits_{\Omega}u^2dx-\frac{1}{2^{\ast}}\int\limits_{\Omega}|u|^{2^{\ast}}dx.
\end{equation*}
In this paper, we obtain  the limit behaviour of $u_\alpha$ ( $\alpha\rightarrow 1$),  energy identity, Pohozaev identity, some integral estimates, etc. And using these results, we prove
infinitely many solutions for the following  Brezis-Nirenberg problem for $N\geq 7$:
\begin{equation*}
\left\{
\begin{aligned} 
 &-\Delta u=|u|^{2^{\ast}-2}u+\lambda u\ \ \ \mbox{in}\ \Omega,\\
&u=0,\ \  \mbox{on}\ \partial\Omega.
\end{aligned}
\right.
\end{equation*}


\section{Introduction}
We know that the energy of harmonic maps does not satisfy the Palais-Smale condition (see \cite{wang1,wang2,zhu1,rsa}).  So, from the viewpoint of calculus
of variation, it is difficult to show the existence of harmonic maps from a surface. In order
to obtain harmonic maps,   Sacks and Uhlenbeck  in \cite{rsa} introduced the so called $\alpha$-energy $E_{\alpha}$
instead of $L^2$ energy $E$ as the following

$$E_{\alpha}=\frac{1}{2\alpha}\int\limits_{\sum} [(1+|\nabla u|^2)^{\alpha}-1 ]dV_g,$$
where $\alpha>1$, $(\sum, g)$ is a Riemann surface, $(N, h)$ is an $n$-dimensional smooth compact Riemannian
manifold which is embedded in $\mathbb{R}^k$ and $u$ is a map  between $\sum$ and $N$.
Using $\alpha$-energy $E_{\alpha}$  Sacks and Uhlenbeck proved that there is a sequence such
that $u_{\alpha}$ converges to a harmonic map $u_1$  outside a finite set of points $X$, as  $\alpha\rightarrow1$.
And the  energy identity of a sequence of $u_{\alpha}$ was consedered in \cite{wang1,wang2,zhu1}.

 Motivated by the ideas of   Sacks and Uhlenbeck, we consider following boundary value
 problem
\begin{equation} \label{P}
\left\{
\begin{aligned} 
 &-\mbox{div}((1+|\nabla u|^2)^{\alpha-1}\nabla u)=|u|^{2^{\ast}-2}u+\lambda u\ \ \ \mbox{in}\ \Omega\\
&u=0,\ \  \mbox{on}\ \partial\Omega,
\end{aligned}
\right.
\end{equation}
where $\alpha>1$, $2^{\ast}=\frac{2N}{N-2}$ and $\Omega$ is a bounded smooth domain in $\mathbb{R}^N$. We call
the nondegenerate operator $-\mbox{div}((1+|\nabla \cdot|^2)^{\alpha-1}\nabla \cdot)$ Sacks-Uhlenbeck operator.
 The energy functional of problem \eqref{P} is
\begin{equation*}
I_{\alpha,\lambda}(u)=\frac{1}{2\alpha}\int\limits_{\Omega} [(1+|\nabla u|^2)^{\alpha}-1 ]dx-\frac{\lambda}{2}\int\limits_{\Omega}u^2dx-\frac{1}{2^{\ast}}\int\limits_{\Omega}|u|^{2^{\ast}}dx, \ \  u\in  H_{0}^{1,2\alpha}(\Omega)
\end{equation*}
and  limit problem of  \eqref{P}  is the well-known Brezis-Nirenberg problem
\begin{equation}\label{app18}
\left\{
\begin{array}{ll}
-\Delta u=|u|^{2^{\ast}-2}u+\lambda u, &\mbox{in}\ \Omega,    \\
u= 0,  &\mbox{on}\  \partial\Omega.
\end{array}
\right.
\end{equation}
Since the embedding $H_{0}^{1,2\alpha}(\Omega)\hookrightarrow L^{2^{\ast}}(\Omega)$
is compact,  we easily show that the functional $I_{\alpha,\lambda}(u)$ satisfies the Palais-Smale condition, then
by symetric Mountain Pass theorem (see \cite{rab}), the functional $I_{\alpha,\lambda}(u)$ has infinitely many
critical points $u_{\alpha,k}, k=1,2,\cdots.$.

Our aim is  to obtain the limit behaviour of $u_{\alpha,k}$ ( $\alpha\rightarrow 1$),  energy identity, Pohozaev identity,  some integral estimates, etc.
And using these results, we prove
infinitely many solutions for  Brezis-Nirenberg problem for $N\geq 7$.

Our  first  result is  the  following  energy identity,:

\begin{theorem}\label{p21}
  Let $\{u_\alpha\}$ be a critical point  of the functional $I_{\alpha}$, and  $u_\alpha\rightharpoonup u$ in $H_{0}^{1}(\Omega)$ as $\alpha\rightarrow 1$. Then there exists a nonnegative integer $k$, $x_{\alpha,j}\in \Omega$, $R_{\alpha,j}\in \mathbb{R}_{+}$, $U_j\in \mathcal{D}^{1,2}(\mathbb{R}^{N})$ such that
  \begin{equation}\label{e8}
    \left\|u_\alpha-u-\sum_{j=1}^k(R_{\alpha,j})^{\frac{N-2}{2}}U_{j}\left(R_{\alpha,j}(x-x_{\alpha,j})\right)\right\|=o(1),
  \end{equation}
\begin{equation}\label{e9}
    I_{\alpha,\lambda}(u_\alpha)=I_{\alpha,\lambda}(u)+\sum_{j=1}^{k} I_{1,0}(U_j, \mathbb{R}^{N})+o(1),
  \end{equation}

  \begin{equation}\label{e10}
R_{\alpha,j} \emph{dist} (x_{\alpha,j},\partial \Omega)\rightarrow +\infty,
  \end{equation}
 where $U_j$ is a positive solution of

 \begin{equation}
\begin{array}{ll}
-\Delta u=u^{2^{\ast}-1}, &\mbox{in}\  \mathbb{R}^N
\end{array}
  \label{P3}
\end{equation}
 and
 \begin{equation}\label{e100}
  I_{1,0}(u, \mathbb{R}^{N})=\frac{1}{2}\int\limits_{\mathbb{R}^{N}} |\nabla u|^2 dx-\frac{1}{2^{\ast}}\int\limits_{\mathbb{R}^{N}}|u|^{2^{\ast}}dx.
\end{equation}
 \end{theorem}

As an application, we use Theorem \eqref{p21} to  prove the existence of  infinitely many solutions for the   well-known Brezis-Nirenberg  problem:
\begin{theorem}\label{rrr}
Suppose that $\lambda >0,N \geq 7$. Then (\ref{app18}) has infinitely many solutions.
\end{theorem}
We know that the functional corresponding
to (\ref{app18}) does not satisfy the Palais-Smale condition at large energy level. So it is
impossible to apply the Mountain Pass lemma directly to obtain the existence of solutions for (\ref{app18}).
The pioneering paper on problem \eqref{e100} was by
Br\'{e}zis-Nirenberg \cite{r8} in 1983 where the authors showed that for $N \geq 4$ and $\lambda\in (0, \lambda_1)$
problem \eqref{e100} has at least one positive solution, where $\lambda_1$ denotes the principal
eigenvalue of $-\Delta$ on $\Omega$.
The same conclusion was proved in  \cite{r8}  for $N = 3$ when
$\Omega$ is a ball and $\lambda \in (\frac{\lambda_1}{4}, \lambda_1)$. In this case, by using the Pohozaev
identity, equation \eqref{e100} has no radial solution when $\lambda \in (0,\frac{\lambda_1}{4})$.
Note that, using the Pohozaev
identity,  \eqref{e100}   has no nontrivial solution when $\lambda \leq 0$ and
$\Omega$ is star-shaped.

Since 1983, there has been a considerable number of papers on
problem \eqref{app18}. Let us now briefly enumerate the multiplicity  results obtained to date  as follows:
\begin{itemize}\addtolength{\itemsep}{-1.5 em} \setlength{\itemsep}{-5pt}
\item[(1)]
Cerami et al. in \cite{cer} proved that
the number of solutions of \eqref{app18} is bounded below by the number of eigenvalues of
$(-\Delta, \Omega)$ lying in the open interval $(\lambda,\lambda+S|\Omega|^{-\frac{2}{N}})$,
where $S$ is the best constant
for the Sobolev embedding $D^{1,2}(\mathbb{R}^N) \rightarrow  L^2(\mathbb{R}^N)$
and $|\Omega|$ is the Lebesgue measure
of $\Omega$.

\item[(2)]  If $N \geq  4$ and $\Omega$ is a ball,
then for any $\lambda > 0$,  infinitely
many sign-changing solutions which were built using particular symmetries of the
domain $\Omega$ were obtained by  Fortunato and Jannelli (see  \cite{for}).

\item[(3)]
In Cerami et al.\cite{cer1}
 it was proved for $N \geq 6$, that \eqref{app18} has two pairs of solutions on any smooth
bounded domain.

\item[(4)]
 Using  Pohozaev identity and  the global compact result Devillanova and Solimini \cite{r3} showed that, if $N \geq 7$, problem
\eqref{app18} has infinitely many solutions for each $\lambda > 0$.
For low dimensions,
that is, $N = 4, 5, 6$, in \cite{devv}, Devillanova and
Solimini proved the existence of at least $N +1$ pairs of solutions provided $\lambda$ is small
enough. In \cite{cla}, Clapp and Weth extended this last result to all $\lambda >0$.

\item[(5)]
Schechter and Zou \cite{sch} showed that in any bounded and
smooth domain, for $N \geq  7 $ and for each fixed $\lambda > 0$, problem \eqref{app18}  has infinitely
many sign changing solutions
\end{itemize}

Using the methods of   \cite{r3},  Cao  et al. in \cite{r9} obtained infinitely many solutions
for  the semilinear elliptic equations involving Hardy potentials and critical Sobolev exponents for $N\geq 7$,
The result of the existence of the infinitely many solutions  was also extended
to  $p$-Laplacian equation ($1 < p < \infty$)  with critical growth for $N > p^2 + p$ in \cite{cao}.

In \cite{r3}, the  well known  global compactness result which gives
a complete description of the  noncompact (P.S.)c sequence
for all energy levels $c$ of    the functional
was used to obtain
 to prove the existence of the infinitely many solutions for \eqref{app18}.
The global compactness result was firstly obtained for  Brezis-Nirenberg problem by M. Struwe \cite{r12}.
For $p$-Laplacian case,    C. Mercuri and  M. Willem  \cite{r13} obtained the global compactness result for all $1<p<N$,
And  the result was proved in \cite{r14} for  singular elliptic problem.  When $\alpha\rightarrow 1$, the solution $u_{\alpha}$
of problem \eqref{P} converges weakly  to $u$, and $u$ is a solution of \eqref{app18}.
Our Theorem \ref{p21} describes  the limit behaviour of $u_\alpha$, and is similar to the global compactness result.

Note that  Theorem \ref{rrr} is similar to the result which was obtained by Devillanova and Solimini \cite{r3}. In \cite{r3},
they firstly  considered the following approximation problem
\begin{equation}\label{app180}
\left\{
\begin{array}{ll}
-\Delta u=|u|^{2^{\ast}-2-\varepsilon}u+\lambda u, &\mbox{in}\ \Omega,    \\
u= 0,  &\mbox{on}\  \partial\Omega,
\end{array}
\right.
\end{equation}
where $\varepsilon>0$, and then, they set up  the  global compactness result, Pohozaev identity, some integral estimates for the approximation problem. 
So they
used  these results to prove  infinitely many solutions  for  Brezis-Nirenberg problem for $N\geq 7$.
In this work,  our problem \eqref{P} can be regarded as  the  approximation problem of    Brezis-Nirenberg problem. So using
the methods of  \cite{r3}, we can also obtain infinitely many solutions   for  Brezis-Nirenberg problem for $N\geq 7$. Since
 the two  approximation problems are different, the two sets  of infinitely many solutions may also be different.
 In a forthcoming paper, we will prove the conclusion.

Let $P(t)=\frac{1}{2\alpha}\left((1+t^2)^{\alpha}-1\right)$, then  $p(t)=P'(t)=(1+t^2)^{\alpha-1}t$
and  $P(t)$ and $p(t)$ satisfy the following inequalities
\begin{align}\label{app17}
&(1):\ 2\leq \frac{tp(t)}{P(t)}\leq 2\alpha,  \ \mbox{for}\  t>0, \nonumber \\
&(2):\ t^{2\alpha}P(u)<P(tu)\leq t^2P(u), \ \ \mbox{if} \ \  0\leq t\leq 1, \\
&(3): \ t^2P(u)\leq  P(tu)\leq t^{2\alpha}P(u), \ \ \mbox{if}  \ \ \ \  t\geq 1.
 \qquad \qquad \qquad\qquad\qquad\qquad\qquad \qquad\quad\qquad\nonumber
\end{align}
Since the operator   $-\mbox{div}((1+|\nabla \cdot|^2)^{\alpha-1}\nabla \cdot)$
is  inhomogeneous, \eqref{app17} is  used to overcome the
nonhomogeneous difficulty.

 \section{Proof of Theorem \ref{p21}}
The Hilbert space $D^{1,2}(\mathbb{R}^{N})$
is the completion of the space  $C_0^{\infty}(\mathbb{R}^{N})$
with respect to the norm
\begin{equation}\label{e3}
  \|u\|_{\mathcal{D}^{1,2}(\mathbb{R}^{N})}=\left\{\int\limits_{\mathbb{R}^{N}}|\nabla u|^2dx\right\}^{\frac{1}{2}}.
\end{equation}
The well-known  Sobolev inequalities  state that for  all $u\in \mathcal{D}^{1,2}(\mathbb{R}^{N})$,
\begin{equation}\label{e4}
  \left(\int\limits_{\mathbb{R}^{N}}|u(x)|^{\frac{2N}{N-2}}dx\right)^{\frac{N-2}{2N}}\leq C\left(\int\limits_{\mathbb{R}^{N}} |\nabla u(x)|^2dx,
  \right)^{\frac{1}{2}}
\end{equation}
where $C$ depends only on $N$.

Set

\begin{equation}\label{e400}
  S_0=\inf\left\{\frac{\int\limits_{\mathbb{R}^{N}} |\nabla u(x)|^2dx}{\left(\int\limits_{\mathbb{R}^{N}}|u(x)|^{2^{\ast}}\right)^{\frac{2}{2^{\ast}}}}: w\in \mathcal{D}^{1,2}(\mathbb{R}^{N})\right\}.
\end{equation}
It is well-known that $S_0$ is attained by the extremal functions
\begin{equation}\label{e5}
  U_{\varepsilon}(x)= \left(\frac{\varepsilon}{x^2+\varepsilon^2}\right)^{\frac{N-2}{2}},
\end{equation}
where $\varepsilon>0$ is arbitrary.

Now using the proof of the  concentration-compactness principle in Orlicz space in \cite{r122},   we can  prove the following lemma:

\begin{lemma}\label{l21}
 Let $\{u_\alpha\}$ be a bounded sequence in  $\mathcal{D}^{1,2}(\mathbb{R}^{N})$. We may assume $u_\alpha$ converges a.e. $u\in \mathcal{D}^{1,2}(\mathbb{R}^{N})$,
 $(1+|\nabla u_{\alpha}|^2)^{\alpha-1}|\nabla u_\alpha|^2$, $u_\alpha^{2^{\ast}}$ converges weakly to some bounded, nonnegative measures $\mu$, $\nu$ on $\mathbb{R}^{N}$, as $\alpha\rightarrow 1$.
  \begin{compactitem}
  \item[\quad\emph{(1):}]  Then we have for some at most countable family $J$, for some families $\{x_j\}_{j\in J}$ for distinct points in $\mathbb{R}^{N}$, $\{\nu_j\}_{j\in J}$ in $(0, \infty)$

      \begin{equation}\label{e6}
        \nu=|u|^{2^{\ast}}+\sum_{j\in J}\nu_j\delta_{x_j},
      \end{equation}
       \begin{equation}\label{e7}
        \mu\geq |\nabla u|^2+\sum_{j\in J} S_0\nu_j^{\frac{2}{2^{\ast}}}\delta_{x_j}.
      \end{equation}
   \item[\quad\emph{(2):}] If $\mu\equiv0$ and $\mu(\mathbb{R}^{N})\leq S_0\nu(\mathbb{R}^{N})^{2/2^{\ast}}$, then $J=\{x_0\}$ for some
   $x_0\in \mathbb{R}^{N}$ and $\nu=c_0\delta_{x_0}$, $\mu=S_0c_0^{2/2^{\ast}}\delta_{x_0}$ for some $c_0>0$.
   \end{compactitem}
\end{lemma}

\begin{lemma}\label{l22}
  Let $\{u_\alpha\}$ be a critical point  of the functional $I_{\alpha}$, and  $u_\alpha\rightharpoonup u$ in $H_{0}^{1}(\Omega)$.
  Then there exist at most finitely many points $x_1,x_2, \cdots x_l\in\Omega$ such that
  $$u_\alpha \rightarrow u \ \mbox{in}\ H_{0}^{1}(\Omega\setminus\{x_1,x_2, \cdots x_l\}.$$
  \end{lemma}

\begin{proof}
Since $\{\nabla u_\alpha\}$ is bounded in $L^{2}(C)$, there exists a $T\in L^{2}(\Omega)$  such that
$$(1+|\nabla u_{\alpha}|^2)^{\alpha-1}\nabla u_\alpha\rightharpoonup T\ \mbox{in}\ L^{2}(\Omega).$$
Obviously, $T$ satisfies
\begin{equation}\label{e11}
\int\limits_{\Omega} T \nabla \psi dx=\int\limits_{\Omega}(|u|^{2^{\ast}-1}\psi+\lambda u\psi)dx, \forall \psi \in H_{0}^{1}(\Omega).
\end{equation}
 By Lemma \ref{l21}, assume that there exists a countable set $\{x_1, x_2, \cdots, x_i, \cdots| x_i\in \Omega\}$, $\nu_j>0$ such that
  \begin{equation}\label{e12}
        |u_\alpha|^{2^{\ast}}\rightharpoonup\nu=|u|^{2^{\ast}}(x, 0)+\sum_{j\in J}\nu_j\delta_{x_j},
      \end{equation}
       \begin{equation}\label{e13}
       (1+|\nabla u_{\alpha}|^2)^{\alpha-1}|\nabla u_\alpha|^2  \rightharpoonup\mu\geq |\nabla u|^2+\sum_{j\in J} S_0\nu_j^{\frac{2}{2^{\ast}}}\delta_{x_j}.
      \end{equation}
 Now we show that $X:=\{x_1, x_2, \cdots\}$ is  a finite set. We choose $\psi=\phi u_\alpha$ in $\langle I'_{\lambda}(u_\alpha),\psi\rangle=o(1)\|\psi\|$,
and let $m\rightarrow +\infty$.  Then by \eqref{e11}, we have
\begin{align*}
  \int\limits_{\Omega}\phi d\nu=  \int\limits_{\Omega} \phi d\mu -\int\limits_{\Omega}\phi T_iD_i u  dx +\int\limits_{\Omega} |u|^{2^{\ast}}\phi dx.
\end{align*}
Let  $\phi$ concentrate on $x_i$, then $\mu(\{x_i\})=\nu(\{x_i\})$.  Moreover, according to
 the relation $\mu(\{x_j\})\geq S_0(\nu(\{x_j\}))^{2/2^{{\ast}}}$, if $\nu_j=\nu(\{x_j\})>0$, then $\nu_j\geq S_0^{N/2}$. Note that $\nu(\bar{\Omega})<\infty$. This implies that $X$ is a finite set.

Choose a function $\varphi\in  C_{0}^{\infty}(\bar{\Omega})$,   $\varphi\geq 0$,
$\varphi\equiv 0$ on $\partial\Omega$,  and $\varphi(x_j)=0$, $\forall x_j\in X$. We get
$$\int\limits_{\Omega} |\varphi u_\alpha|^{2^{\ast}}dx \rightarrow \int\limits_{\Omega} |\varphi u|^{2^{\ast}}dx+\sum_{j}\nu_j\varphi^{2^{\ast}}(x_j)= \int\limits_{\Omega} |\varphi u|^{2^{\ast}}dx.$$
So,   $\varphi u_\alpha\rightarrow \varphi u$  in $L^{2^{\ast}}(\Omega)$. Furthermore,  we easily  obtain
\begin{align}\label{e144}
 &\int\limits_{\Omega} \varphi(1+|\nabla u_{\alpha}|^2)^{\alpha-1}\nabla u_{\alpha}\nabla ( u_{\alpha}-u)dx-\int\limits_{\Omega} \varphi\nabla u\nabla ( u_{\alpha}-u)dx\nonumber \\
 &=\int\limits_{\Omega}(|u_\alpha|^{2^{\ast}-1}-|u|^{2^{\ast}-1} )(u_\alpha-u)\varphi dx
   +\lambda \int\limits_{\Omega} \varphi  (u_\alpha-u)^2dx \nonumber \\
 & -\int\limits_{\Omega} (1+|\nabla u_{\alpha}|^2)^{\alpha-1}\nabla u_{\alpha}\nabla \varphi ( u_{\alpha}-u)dx-\int\limits_{\Omega} \nabla u\nabla\varphi ( u_{\alpha}-u)dx +o(1) \rightarrow 0\ \mbox{as}\ n\rightarrow \infty.
\end{align}
We easily know that $$\int\limits_{\Omega} \varphi(1+|\nabla u_{\alpha}|^2)^{\alpha-1}\nabla u_{\alpha}\nabla ( u_{\alpha}-u)dx-\int\limits_{\Omega} \varphi\nabla u\nabla ( u_{\alpha}-u)dx\backsim \int\limits_{\Omega} \varphi (|\nabla u_{\alpha}|^{2}- |\nabla u|^2)dx$$

So, we have
\begin{align}\label{e14}
 &\int\limits_{\Omega} \varphi (|\nabla u_{\alpha}|^{2}- |\nabla u|^2)dx
 \backsim \int\limits_{\Omega}(|u_\alpha|^{2^{\ast}-1}-|u|^{2^{\ast}-1} )(u_\alpha-u)\varphi dx
   +\lambda \int\limits_{\Omega} \varphi  (u_\alpha-u)^2dx \nonumber \\
 & -\int\limits_{\Omega} (1+|\nabla u_{\alpha}|^2)^{\alpha-1}\nabla u\nabla \varphi ( u_{\alpha}-u)dx-\int\limits_{\Omega} \nabla u\nabla\varphi ( u_{\alpha}-u)dx +o(1) \rightarrow 0\ \mbox{as}\ n\rightarrow \infty.
\end{align}
This implies that
$$u_\alpha \rightarrow u \ \mbox{in}\ H_{0,  loc}^{1}(\Omega\setminus\{x_1,x_2, \cdots x_l\}).$$
\end{proof}

\begin{lemma}\label{l23}
  Suppose that $u_{\alpha}\in H_{0}^{1}(\Omega)$ satisfy $I_{\alpha,0}'(u_{\alpha})\rightarrow 0$ and $u_{\alpha}\rightharpoonup 0$ in $H_{0}^{1}(\Omega)$. Then there exist three sequences  $\{x_{\alpha}\}\subset \Omega$, $R_{\alpha}\rightarrow\infty$ and $w_{\alpha}\in H_{0}^{1}(\Omega)$ such that
  $$w_{\alpha}(x)=u_{\alpha}-R_{\alpha}^{(N-2)/2}U_0(R_{\alpha}(x-x_{\alpha}))+o(1),$$
  $$|I_{\alpha, 0}(u_{\alpha})-I_{\alpha, 0}(w_{\alpha})-I_{\alpha, 0}(U_0, \mathbb{R}^{N})|\rightarrow 0,$$
  $$I'_{\alpha, 0}(w_{\alpha})\rightarrow 0,$$
  $$R_{\alpha} \emph{dist}(x_{\alpha}, \partial \Omega)\rightarrow +\infty,$$
where $o(1)\rightarrow 0$ in $\mathcal{D}^{1,2}(\mathbb{R}^{N+1})$ and $U_0$ is a solution of
\begin{equation}\label{P5}
-\Delta u=u^{2^{\ast}-1}, \mbox{in}\  \mathbb{R}_{+}^N.
\end{equation}
\end{lemma}

\begin{proof}
We assume that
\begin{equation}\label{e15}
  |u_{\alpha}|^{2^{\ast}}\rightarrow \sum \nu_j\delta_{x_j}, x_j \in \bar{\Omega}.
\end{equation}
Then there  exists at least one $\nu_j\not=0$. Otherwise, $u_{\alpha}\rightarrow 0$ in $L^{2^{\ast}}(\Omega)$. And by $I_{\alpha,0}'(u_{\alpha})\rightarrow 0$, we have
 $u_{\alpha}\rightarrow 0$ in $H_{0}^{1}(\Omega)$.

 Set

 \begin{equation}\label{e16}
   Q_{\alpha}(r)=\sup_{x\in \bar{\Omega}}\left\{\int\limits_{B_r(x)}|\nabla u_{\alpha}|^2dx+\int\limits_{B_r(x)}|u_{\alpha}|^{2^{\ast}}dx\right\}.
 \end{equation}
For sufficiently small $\tau \in(0, S_0^{N/2})$, choose $R_{\alpha}=R_{\alpha}(\tau)>0, x_{\alpha}\in \bar{\Omega}$ such that

 \begin{equation}\label{e17}
   \left\{\int\limits_{B_{1/R_{\alpha}}(x_{\alpha})}|\nabla u_{\alpha}|^2dx+\int\limits_{B_{1/R_{\alpha}}(x_{\alpha})}|u_{\alpha}|^{2^{\ast}}dx\right\}=Q_{\alpha}\left(\frac{1}{R_{\alpha}}\right)=\tau.
 \end{equation}
Set
\begin{equation}\label{e18}
  \tilde{u}_{\alpha}(x)=R_{\alpha}^{(2-N)/2}u_{\alpha}\left(\frac{x}{R_{\alpha}}+x_{\alpha}\right), x\in \Omega_{\alpha}=\left\{x: \frac{x}{R_{\alpha}}+x_{\alpha}\in \Omega\right\}.
\end{equation}
Then,  we obtain
 \begin{align}\label{e19}
   \tilde{Q}_{\alpha}(r)&=\sup_{x\in \mathbb{R}^{N}}\left\{\int\limits_{B_r(x)}|\nabla \tilde{u}_{\alpha}|^2dx+\int\limits_{B_r(x)}|\tilde{u}_{\alpha}|^{2^{\ast}}dx\right\}\nonumber\\
   &=\sup_{x\in \mathbb{R}^{N}}\left\{\int\limits_{B_{r/R_{\alpha}}(x)}|\nabla u_{\alpha}|^2dx+\int\limits_{B_{r/R_{\alpha}}(x)}|u_{\alpha}|^{2^{\ast}}dx\right\}\nonumber\\
   &=Q_\alpha\left(\frac{r}{R_{\alpha}}\right).
 \end{align}
Hence,
\begin{equation}\label{e20}
  \left\{\int\limits_{B_1(0)}|\nabla \tilde{u}_\alpha|^2dx+\int\limits_{B_1(0)}|\tilde{u}_\alpha|^{2^{\ast}}dx\right\}=\tilde{Q}_\alpha(1)=\tau.
\end{equation}
Now we will show that there exists  a sufficiently small $\tau\in (0, S_0^{N/2})$ such that $R_{\alpha}(\tau)\rightarrow +\infty$, as $m\rightarrow +\infty$,
if not, for each $\varepsilon>0$, there is a constant $M_{\varepsilon}>0$ such that $R_{\alpha}(\varepsilon)\leq M_{\varepsilon}$. So,
\begin{align}\label{e21}
 \int\limits_{B_{1/M_\varepsilon}(x)}|u_{\alpha}|^{2^{\ast}}dx&\leq  \sup_{x\in \bar{\Omega}} \left\{\int\limits_{B_{1/R_{\alpha}}(x_{\alpha})}|\nabla u_{\alpha}|^2dx+\int\limits_{B_{1/R_{\alpha}}(x_{\alpha})}|u_{\alpha}|^{2^{\ast}}dx\right\}\nonumber\\
&=Q_{\alpha}\left(\frac{1}{R_{\alpha}}\right)=\varepsilon, \forall x\in \bar{\Omega}.
\end{align}
Futhermore,
$$\nu_j\leq \int\limits_{B_{1/M_\varepsilon}(x)}|u_{\alpha}|^{2^{\ast}}dx +o(1)\leq \varepsilon+o(1), \forall \varepsilon >0.$$
Contradiction!

Now  we distinguish two case:

(i) $R_{\alpha} \mbox{dist}(x_{\alpha}, \partial \Omega)\rightarrow +\infty$, in this case $\Omega_\alpha\rightarrow \Omega^{\infty}=\mathbb{R}^{N}$.

(ii) $R_{\alpha} \mbox{dist}(x_{\alpha}, \partial \Omega)\rightarrow M<+\infty$, uniformly. Then after an orthogonal transformation,
$$\Omega_\alpha\rightarrow \Omega^{\infty}=\mathbb{R}_{+}^{N}=\{x=(x_1,  \cdots, x_N ), x_1>0\}.$$

If $(x, y)\not \in \Omega_\alpha$, we define $\tilde{u}_{\alpha}=0$. Since
$$\int\limits_{\mathbb{R}^{N}}|\nabla \tilde{u}_\alpha|^2 dx=\int\limits_{\Omega_\alpha} |\nabla \tilde{u}_\alpha|^2 dx=\int\limits_{\Omega} |\nabla u_{\alpha}|^2 dx,$$
we can assume that $\tilde{u}_\alpha\rightharpoonup U_0$ in $\mathcal{D}^{1,2}(\mathbb{R}_{+}^{N})$.

Since in each case for  any $\varphi \in C_{}^{\infty}(\bar{\Omega}^{\infty})$,
we have that $\varphi \in  C_{}^{\infty}(\bar{\Omega}_{\alpha})$,
for large $m$, there holds
\begin{align}\label{e22}
  & \int\limits_{\mathbb{R}^{N}} \nabla \tilde{u}_{\alpha} \nabla \varphi dx-\int\limits_{\mathbb{R}^N} |\tilde{u}_{\alpha}|^{2^{\ast}-1}\phi dx\nonumber \\
  & = \int\limits_{\mathbb{R}^{N}} \nabla u_{\alpha} \nabla \varphi_{\alpha}^{\ast} dx-\int\limits_{\mathbb{R}^N} |u_{\alpha}|^{2^{\ast}-1}\varphi_{\alpha}^{\ast} dx\nonumber \\
  & = \int\limits_{\mathbb{R}^{N}}(1+|\nabla \tilde{u}_{\alpha}|^2)^{\alpha-1} \nabla u_{\alpha} \nabla \varphi_{\alpha}^{\ast} dx-\int\limits_{\mathbb{R}^N} |u_{\alpha}|^{2^{\ast}-1}\varphi_{\alpha}^{\ast} dx+o(1)\nonumber \\
  &= o(1) \|\varphi_{\alpha}^{\ast}\|_{H_{}^{1}(C)}= o(1) \|\varphi\|_{\mathcal{D}^{1,2}(\mathbb{R}^{N})},
\end{align}
where $\varphi_{m}^{\ast}(x)=R_{\alpha}^{(2-p)/2}\varphi(R_{\alpha}(x-x_{\alpha}))\in H_{0}^1(\Omega)$. We are going to  prove that
 \begin{equation}\label{e225}
 \tilde{u}_{\alpha}\rightarrow U_0\ \ \mbox{in}\ \ \mathcal{D}_{loc}^{1,2}(\mathbb{R}^{N}).
 \end{equation}
To do this, by \eqref{e22}, we next  only need to show that
$$\tilde{u}_{\alpha}\rightarrow U_0\ \ \mbox{in}\ \ L_{loc}^{2^{\ast}}(\mathbb{R}^{N}).$$
Indeed, choose a smooth  cut-off function $\eta\in  H_{0}^{1}(\Omega)$ such that $0\leq \eta \leq 1$,  $\eta=1$ in $B_1$, $\eta=0$ outside $B_2$, where,
$B_{\rho}=\{(x): |(x|\leq \rho\}.$
Assume that
  \begin{equation}\label{e23}
        |\eta \tilde{u}_{\alpha}|^{2^{\ast}}\rightharpoonup\nu=|\eta U_0|^{2^{\ast}}(x, 0)+\sum_{j\in J}\nu_j\delta_{x_j},
      \end{equation}
   \begin{equation}\label{e24}
       |  \nabla(\eta \tilde{u}_{\alpha})|^2  \rightharpoonup\mu\geq |\nabla (\eta U_0)|^2+\sum_{j\in J} S_0\nu_j^{\frac{2}{2^{\ast}}}\delta_{x_j}.
      \end{equation}
Similar to Lemma \ref{l21}, it is easy to show that if $\nu_j=\nu(\{x_j\})>0$, then  $\nu_j\geq S_0^{N/2}$, where $x_j\in B_k$. Moreover
$$S_0^{N/2}>\tau=\tilde{Q}_{\alpha}(1)\geq \int\limits_{B_1(x_j)}|\nabla \tilde{u}_{\alpha}|^2dx \geq \mu(\{x_j\})\geq S_0^{N/2}.$$
Contradiction. So, $\nu_j=\nu(\{x_j\})=0$, and \eqref{e225} holds.

Next we will show that case (ii) doesn't occur. In fact,  using \eqref{e22} and  \eqref{e225}, one has
\begin{equation}\label{P99}
-\Delta U_0=U_0^{2^{\ast}-1}, \mbox{in}\  \mathbb{R}^N,
\end{equation}
Similar to  the proof in Theorem 1.1  in \cite{r12}, we can use   Pohozaev identity and Strong maximum principle (see \cite{r10}) to show that $U_0=0$. This  contradicts the following relation
\begin{align}
&\int\limits_{B_1(0)} |\nabla U_0|^2dx+\int\limits_{B_1(0)}|U_0|^{2^{\ast}}dx\nonumber\\
&\ \ =\lim_{\alpha\rightarrow 1}\int\limits_{B_1(0)} (1+|\nabla \tilde{u}_{\alpha}|^2)^{\alpha-1}|\nabla \tilde{u}_{\alpha}|^2dx+\int\limits_{B_1(0)}|\tilde{u}_{\alpha}|^{2^{\ast}}dx=\tau>0.
\end{align}
Hence, $R_{\alpha} \mbox{dist}(x_{\alpha}, \partial \Omega)\rightarrow +\infty$.

Let  $\alpha\in C(B_2)$  satisfying $0\leq \alpha\leq 1$, $\alpha=1$ in $B(0,1)$, and $\alpha=0$ outside $B(0,2)$. Set
$$w_{\alpha}(x)=u_{\alpha}(x)-R_{\alpha}^{(N-2)/2}U_0(R_{\alpha}(x-x_{\alpha}))\alpha (\bar{R}_{\alpha} (x-x_{\alpha}))\in H_{0, L}^1(C)),$$
where the sequence  $\bar{R}_{\alpha}$ is chosen such that $\bar{R}_{\alpha} \mbox{dist}(0, \partial \Omega)\rightarrow +\infty$ and
$\tilde{R}_{\alpha}:=\frac{R_{\alpha}}{\bar{R}_{\alpha}}\rightarrow +\infty$,  then  we have
$$\tilde{w}_{\alpha}(x)=\tilde{u}_{\alpha}(x)-U_0(x)\alpha \left(\frac{x}{\tilde{R}_{\alpha}}\right).$$

Similar to \cite{r12}, it is easy to show that
\begin{equation}\label{e24.5}
  w_{\alpha}(x)=u_{\alpha}-R_{\alpha}^{(N-2)/2}U_0(R_{\alpha}(x-x_{\alpha}))+o(1),
\end{equation}
where $o(1)\rightarrow 0$ in $\mathcal{D}^{1,2}(\mathbb{R}^{N+1})$.

\begin{equation}\label{e24.6}
  \tilde{w}_{\alpha}=\tilde{u}_{\alpha}-U_0+o(1).
\end{equation}
Therefore, by \eqref{e225}, we get
\begin{align}\label{e25}
   & \int\limits_{\Omega} [(1+|\nabla u_{\alpha}|^2)^{\alpha}-1 ]dx- \int\limits_{\Omega}[(1+|\nabla w_{\alpha}|^2)^{\alpha}-1 ]dx\nonumber \\
   &\ \ = \int\limits_{\Omega} |\nabla u_{\alpha}|^2dx-\int\limits_{\Omega} |\nabla u_{\alpha}-R_{\alpha}^{(N-2)/2}U_0(R_{\alpha}(x-x_{\alpha}))|^2dx+o(1)\nonumber \\
   &\ \ =\int\limits_{\mathbb{R}^{N}} |\nabla \tilde{u}_{\alpha}|^2dx-\int\limits_{\mathbb{R}^{N}} |\nabla \tilde{u}_{\alpha}-U_0|^2dx+o(1)\nonumber\\
   &\ \ =  \int\limits_{\mathbb{R}^{N}}\int\limits_0^1 2(t\nabla \tilde{u}_{\alpha}+(1-t)\nabla (\tilde{u}_{\alpha}-U_0))\nabla U_0dtdx+o(1)\nonumber \\
   &\ \ \rightarrow  \int\limits_{\mathbb{R}^{N}}\int\limits_0^1 2|\nabla U_0|^2t dtdx=\int\limits_{\mathbb{R}^{N}} |\nabla U_0|^2dx.
\end{align}
To proceed, observe that like (\ref{e25}), we  have
$$\int\limits_{\Omega}|u_{\alpha}|^{2^{\ast}}dx-\int\limits_{\Omega}|w_{\alpha}|^{2^{\ast}}dx\rightarrow \int\limits_{\mathbb{R}^{N} }|u|^{2^{\ast}}dx.$$
So $$|I_{\alpha,0}(u_{\alpha})-I_{\alpha,0}(w_{\alpha})-I_{\alpha,0}(u, \mathbb{R}^{N})|\rightarrow 0.$$
On the other hand,  according to \eqref{e24.5}  and \eqref{e24.6}, we easily infer
\begin{align}\label{e26}
                     &|\langle I_{\alpha, 0}'(u_{\alpha}), \varphi \rangle-\langle I_{\alpha, 0}'(w_{\alpha}), \varphi \rangle| \nonumber \\
                      & \leq  C\left[\left(\int\limits_{\mathbb{R}^{N}}|\nabla \tilde{u}_{\alpha}-\nabla \tilde{w}_{\alpha}-\nabla U_0|^{2}dx\right)^{\frac{1}{2}}\right.\nonumber \\
                      &\ \  + \left.\left(\int\limits_{\mathbb{R}^{N}}\left( |\tilde{u}_{\alpha}|^{2^{\ast}-1}-|\tilde{w}_{\alpha}|^{2^{\ast}-1}-|U_0|^{2^{\ast}-1}\right)^{2^{\ast'}}dx\right)^{\frac{1}{2^{\ast'}}}  \right]\| \varphi \|_{\mathcal{D}^{1,2}(\mathbb{R}_{+}^{N+1})}\rightarrow 0.
                      \end{align}
\end{proof}

\begin{proof}[\textbf{Proof of Theorem \ref{p21}}]
Applying Lemma \ref{l23} to to the sequences,
$$v^1_{\alpha}=u_{\alpha}-u,$$
$$v^j_{\alpha}=u_{\alpha}-u-\sum_{i=1}^{j-1}U_{\alpha}^i=v_{\alpha}^{j-1}-U_{\alpha}^{j-1}, j>1,$$
where $U_{\alpha}^{i}(x)=(R_{\alpha}^i)^{\frac{N-2}{2}}U_{i}\left(R_{\alpha}^i(x-x_n^i)\right)$.

By induction
\begin{align*}
  I_{\alpha, 0}(v_{\alpha}^j) &=I_{\alpha,\lambda}(u_{\alpha})-I_{\alpha,\lambda}(u)-\sum_{i=1}^{j-1}I_{\alpha,0}(U^j,\mathbb{R}^{N}), \\
  &\leq I_{\alpha,\lambda}(u_{\alpha})-(j-1)\beta^{\ast},
\end{align*}
where $\beta^{\ast}=\frac{1}{2N}S^N$.

Since the latter will be negative for large $j$, the iteration must stop after finite steps; moreover, for this index we have
$$v_{m}^{k+1}=u_{\alpha}-u-\sum_{j=1}^kU_{\alpha}^j\rightarrow 0\ \  \mbox{in}\ \ $$
and $$I_{\alpha,\lambda}(u_{\alpha})-I_{\alpha,\lambda}(u)-\sum_{j=1}^kI_{\alpha, 0}(U_j,\mathbb{R}^{N})\rightarrow 0.$$
The proof is complete.
\end{proof}

 \section{Proof of Theorem \ref{rrr}}
In this section, we assume that $1<\alpha<\frac{N}{N-2}$. This section will be  divided into three subsections. In subsection 3.1, we will give some integral estimates. In subsection, we will prove some estimates on safe regions. Subsection 3.3 is devoted to proving Theorem \ref{rrr}.
\subsection{Some integral estimates}
For  any $p_2<2^{\ast}<p_1$, $\beta>0$ and $R\geq 1$, let us consider the following relation:
\begin{equation}\label{a1}
\left\{
\begin{array}{ll}
\|u_1\|_{p_1}\leq \beta,    \\
\|u_2\|_{p_2}\leq \beta R^{\frac{N}{2^{\ast}}-\frac{N}{p_2}}.
\end{array}
\right.
\end{equation}
Based the idea of  \cite{r3}, we define
$$\|u\|_{p_1,p_2,R}=\inf\{\beta: \mbox{there are\ } u_1 \ \mbox{and}\  u_2, \mbox{such that} (\ref{a1}) \mbox{\ holds \ and\ } |u|\leq u_1+u_2 \}.$$
In this subsection, our main result is the following proposition.
\begin{proposition}\label{p31}
Let $u_{\alpha}$ be a weak solution of (\ref{P}) with $\alpha\rightarrow 1$. For any $p_1$, $p_2$ satisfying
$\left(1-\frac{1}{2\alpha}\right)<p_2<2^{\ast}<p_1$,  there exists  a constant $C$, depending on $p_1$ and $p_2$, such that
$$\|u_{\alpha}\|_{p_1,p_2,R_{\alpha}}\leq C.$$
\end{proposition}

\begin{lemma}\label{l31}
Assume that $\Omega_1$ is a bounded domain satisfying $\Omega \subset \Omega_1$. For any functions
$f_1(x) \geq 0$ and $f_2(x) \geq 0$ , let $w \geq 0$  be the solution of
\begin{equation} \label{a2}
\left\{
\begin{aligned} 
 &-\emph{div}((1+|\nabla w|^2)^{\alpha-1}\nabla w)=f_1(x)+f_2(x)\ \ \ \mbox{in}\ \Omega_1,\\
&w=0,\ \  \mbox{on}\ \partial\Omega_1.
\end{aligned}
\right.
\end{equation}
Let $w_i$, $i=1,2$, be the solution of
\begin{equation} \label{ab2}
\left\{
\begin{aligned} 
 &-\emph{div}((1+|\nabla w_i|^2)^{\alpha-1}\nabla w)=f_i(x)\ \ \ \mbox{in}\ \Omega_1,\\
&w_i=0,\ \  \mbox{on}\ \partial\Omega_1.
\end{aligned}
\right.
\end{equation}
Then, there is a constant $C>0$, depending only on $r =\frac{1}{3}\mbox{dist}(\Omega, \partial \Omega_1)$,
such that
\begin{align}\label{ab3}
w(x)\leq  C\inf_{y\in B_r(x_0)}w(y)+Cw_1(x)+Cw_1(x), \mbox{\ for \ all}\ x\in \Omega.
\end{align}

\end{lemma}

\begin{proof}
Let $r =\frac{1}{3}\mbox{dist}(\Omega, \partial \Omega_1)$. Then it follows from \cite[Theorem 2]{kuu}
 that for any $x_0\in \Omega$,
$$w(x_0)\leq C_2\inf_{x\in B_r(x_0)}w(x)+C_3W_{1,2\alpha}(x_0,2r,f_1+f_2),$$
where $W_{1,2\alpha}(x_0,r,f)$ is the Wolff potential for the function $f$:
$$W_{1,2\alpha}(x_0,r,f)=\int\limits_0^r\left(\int\limits_{B_t(x_0)}|f|\right)^{\frac{1}{2\alpha-1}}   \frac{dt}{t^{\frac{N-2\alpha}{2\alpha-1}+1}}.$$
On the other hand, it is easy to show that
$$W_{1,2\alpha}(x_0,2r,f_1+f_2)\leq 2^{\frac{1}{2\alpha-1}}W_{1,2\alpha}(x_0,2r,f_1)+2^{\frac{1}{2\alpha-1}}W_{1,2\alpha}(x_0,2r,f_2).$$
So, using \cite[Theorem 2]{kuu} again, we obtain
\begin{align}\label{a3}
w(x_0)&\leq C_2\inf_{x\in B_r(x_0)}w(x)+C_3W_{1,2\alpha}(x_0,2r,f_1)+C_3W_{1,2\alpha}(x_0,2r,f_2)\nonumber\\
&\leq  C_2\inf_{x\in B_r(x_0)}w(x)+C_4w_1(x_0)+C_5w_1(x_0).
\end{align}

\end{proof}

\begin{lemma}\label{l32}
Let $w$ be the solution of
\begin{equation} \label{a4}
\left\{
\begin{aligned} 
 &-\emph{div}((1+|\nabla w|^2)^{\alpha-1}\nabla w)=a(x)v^{2\alpha-1}\ \ \ \mbox{in}\ \Omega,\\
&w=0,\ \  \mbox{on}\ \partial\Omega,
\end{aligned}
\right.
\end{equation}
where $a(x) \geq 0$ and $v \geq  0$  are functions satisfying $a, v \in  L^{\infty}(\Omega)$. Then, for any $p_1 >2^{\ast}>p_2 >\left(1-\frac{1}{2\alpha}\right)2^{\ast}$,
there is a constant $C = C(p_1,p_2, |\Omega|)$, such that for any $R \geq 1$,
$$\|w\|_{p_1,p_2,R}\leq C\|a\|^{\frac{1}{2\alpha-1}}_{\frac{N}{2\alpha}}\|v\|_{p_1,p_2,R}.$$
\end{lemma}
\begin{proof}
For any small $\theta >0$, let $v_1 \geq 0$ and $v_2 \geq  0$ be the functions such that $v \leq  v_1 + v_2$, and
(\ref{a1}) holds with $\alpha  =	\|v\|_{p_1, p_2, R} +\theta$ . Choose a domain $\Omega_1$  with $\Omega\subset \Omega_1$.
We let $a(x) = 0$ and let
$v_i = 0$ in $\Omega_1\setminus\Omega$.  Consider
\begin{equation} \label{a5}
\left\{
\begin{aligned} 
 &-\mbox{div}((1+|\nabla w_i|^2)^{\alpha-1}\nabla w_i)=a(x)v_i^{2\alpha-1}\ \ \ \mbox{in}\ \Omega_1,\\
&w_i=0,\ \  \mbox{on}\ \partial\Omega_1.
\end{aligned}
\right.
\end{equation}
Then, it follows from Corollary \ref{app3} that
$$\|w_i\|_{p_i}\leq C\|a\|^{\frac{1}{2\alpha-1}}_{\frac{N}{2\alpha}}\|v\|_{p_i},\ \ i=1,2.$$
On the other hand, by the comparison principle, we can deduce $w \leq \tilde{w}$ in $\Omega$, where  $\tilde{w}$ is the
solution of
\begin{equation} \label{a6}
\left\{
\begin{aligned} 
 &-\mbox{div}((1+|\nabla \tilde{w}|^2)^{\alpha-1}\nabla \tilde{w})=2^{2\alpha-1}a(x)(v_1^{2\alpha-1}+v_2^{2\alpha-1}),\ \ \ \mbox{in}\ \Omega,\\
&\tilde{w}=0,\ \  \mbox{on}\ \partial\Omega.
\end{aligned}
\right.
\end{equation}
It follows from Corollary \ref{app3} again that
\begin{align}\label{a7}
\|\tilde{w}\|_{p_2}&\leq C\|a\|^{\frac{1}{2\alpha-1}}_{\frac{N}{2\alpha}}\|(v_1^{2\alpha-1}+v_2^{2\alpha-1})^{\frac{1}{2\alpha-1}}\|_{p_2}\nonumber\\
&\leq C\|a\|^{\frac{1}{2\alpha-1}}_{\frac{N}{2\alpha}}\|(\|v_1\|_{p_2}+\|v_2\|_{p_2})\nonumber\\
&\leq C\|a\|^{\frac{1}{2\alpha-1}}_{\frac{N}{2\alpha}}\|(\|v_1\|_{p_1}+\|v_2\|_{p_2}).
\end{align}
As before, let $r = \frac{1}{3}\mbox{dist}(\Omega, \partial \Omega_1)$. So, we obtain
\begin{align}\label{a8}
   \inf_{x\in B_r(x_0)} \tilde{w}(x)&\leq Cr^{-N}\int\limits_{B_{r}(x_0)}  \tilde{w}dx\leq Cr^{-\frac{N}{p_2}}\|\tilde{w}\|_{p_2} \nonumber \\
   & \leq C\|a\|^{\frac{1}{2\alpha-1}}_{\frac{N}{2\alpha}}\|(\|v_1\|_{p_1}+\|v_2\|_{p_2}) \ \ \mbox{for\ all } x_0\in \Omega.
\end{align}
On the other hand, it follows from Lemma \ref{l31} that there is $C_1 > 0$, such that
\begin{align}\label{a9}
\tilde{w}(x_0)&\leq C \inf_{x\in B_r(x_0)} \tilde{w}(x)+Cw_1(x_0)+Cw_2(x_0) \nonumber \\
   & \leq C\|a\|^{\frac{1}{2\alpha-1}}_{\frac{N}{2\alpha}}\|(\|v_1\|_{p_1}+\|v_2\|_{p_2})+Cw_1(x_0)+Cw_2(x_0).
\end{align}
Combining $w(x) \leq  \tilde{w}(x)$, for all $x \in  \Omega$, and (\ref{a9}), we obtain
\begin{equation}\label{a10}
 w(x)\leq \hat{w}_1(x)+\hat{w}_2(x),\ \mbox{for\ all} x\in \Omega,
\end{equation}
where
$$ \hat{w}_1(x)=C\|a\|^{\frac{1}{2\alpha-1}}_{\frac{N}{2\alpha}}(\|v_1\|_{p_1}+\|v_2\|_{p_2})+C_1w_1(x),$$
and $$\hat{w}_2(x)=C_1w_2(x).$$
Noting that $R^{\frac{N}{2^{\ast}}-\frac{N}{p_2}}\leq 1$  as $R\geq 1$, by \eqref{a7}, we deduce
\begin{align}\label{a11}
\|\hat{w}_1\|_{p_1}&\leq C\|a\|^{\frac{1}{2\alpha-1}}_{\frac{N}{2\alpha}}(\|v_1\|_{p_1}+\|v_2\|_{p_2})+C\|w_1\|_{p_1}\nonumber\\
&\leq C'\|a\|^{\frac{1}{2\alpha-1}}_{\frac{N}{2\alpha}}(\|v_1\|_{p_1}+\|v_2\|_{p_2})\nonumber\\
&\leq C\|a\|^{\frac{1}{2\alpha-1}}_{\frac{N}{2\alpha}}(\|v_1\|_{p_1,p_2,R}+\theta),
\end{align}
and
\begin{align}\label{a12}
\|w_2\|_{p_2}&\leq C\|a\|^{\frac{1}{2\alpha-1}}_{\frac{N}{2\alpha}}\|v_2\|_{p_2}
\leq C'\lambda^{\frac{N}{2^{\ast}}-\frac{N}{p_2}}\|a\|^{\frac{1}{2\alpha-1}}_{\frac{N}{2\alpha}}(\|v_1\|_{p_1,p_2,R}+\theta).
\end{align}
So, the result follows.
\end{proof}

\begin{lemma}\label{l34}
Let $w$ be a weak solution of
\begin{equation} \label{a13}
\left\{
\begin{aligned} 
 &-\emph{div}((1+|\nabla w|^2)^{\alpha-1}\nabla w)=2v^{2^{\ast}-1}+A\ \ \ \mbox{in}\ \Omega,\\
&w=0,\ \  \mbox{on}\ \partial\Omega,
\end{aligned}
\right.
\end{equation}
where $v \geq  0$ and $v \in L^{\infty}(\Omega)$.  For any $p_1,p_2$ with  $2^{\ast}-1<p_2 <2^{\ast}<p_1 <\frac{N(2^{\ast}-1)}{2\alpha}$, let
\begin{equation}\label{19}
\frac{1}{q_1}=\frac{2^{\ast}-1}{(2\alpha-1)p_i}-\frac{2\alpha}{N(2\alpha-1)},i=1,2.
\end{equation}
Then there is a constant $C = C(p_1, p_2)$, such that for any $R >0$,

$$\|w\|_{q_1,q_2,R}\leq C\|v\|^{\frac{2^{\ast}-1}{2\alpha-1}}_{p_1,p_2,R}+C.$$
\end{lemma}
\begin{proof}
For any small $\theta >0$, let $v_1 \geq  0$ and $v_2 \geq 0$  be the functions such that $v \leq  v_1 + v_2$, and
(\ref{a1}) holds with $\alpha =	\|v\|_{p_1,p_2,R} +\theta$. Let $\tilde{w}$ be the solution of
\begin{equation} \label{a14}
\left\{
\begin{aligned} 
 &-\mbox{div}((1+|\nabla \tilde{w}|^2)^{\alpha-1}\nabla \tilde{w})=2^{2^{\ast}}v_1^{2^{\ast}-1}+2^{2^{\ast}}v_2^{2^{\ast}-1}+A\ \ \ \mbox{in}\ \Omega_1,\\
&\tilde{w}=0,\ \  \mbox{on}\ \partial\Omega_1.
\end{aligned}
\right.
\end{equation}
Then $w \leq  \tilde{w}$, It follows from Proposition \ref{app1} that there is a $\bar{p} > 0$, such that 	$\|\tilde{w}\|p \leq C$. Thus,
$$\inf_{x\in B_r(x_0)} \tilde{w}(x)\leq C,\ \ \mbox{for\ all} \ x_0\in \Omega$$

Now we consider
\begin{equation} \label{a15}
\left\{
\begin{aligned} 
 &-\mbox{div}((1+|\nabla w_1|^2)^{\alpha-1}\nabla w_1)=2^{2^{\ast}}v_1^{2^{\ast}-1}+A\ \ \ \mbox{in}\ \Omega_1,\\
&w_1=0,\ \  \mbox{on}\ \partial\Omega_1,
\end{aligned}
\right.
\end{equation}
and

\begin{equation} \label{a16}
\left\{
\begin{aligned} 
 &-\mbox{div}((1+|\nabla w_2|^2)^{\alpha-1}\nabla w_2)=2^{2^{\ast}}v_2^{2^{\ast}-1}\ \ \ \mbox{in}\ \Omega_1,\\
& w_2=0,\ \  \mbox{on}\ \partial\Omega_1.
\end{aligned}
\right.
\end{equation}
Then, by Lemma \ref{l31},
$$\tilde{w}(x_0)\leq C +Cw_1(x_0)+Cw_2(x_0) ,\mbox{for\ all}\ x_0\in \Omega.$$

Let $\hat{p}_i=\frac{p_i}{2^{\ast}-1}$, then $q_i=\frac{N(2\alpha-1)\hat{p}_i}{N-2\alpha\hat{p}_i}$, $i=1,2.$
Besides, for $p_i\in \left(2^{\ast}-1,\frac{N(2^{\ast}-1)}{2\alpha}\right)$, we have $\hat{p}_i\in \left(p_i, \frac{N}{2\alpha}\right)$.

By Proposition \ref{app1}, we have

\begin{align}\label{a17}
\|C+Cw_1\|_{q_1}&\leq C'+C\|w_1\|_{q_1} \leq C'+C'\|v^{2^{\ast}-1}_1+A\|_{\hat{p}_1}^{\frac{1}{2\alpha-1}}\leq C\left(1+\|v_1\|_{p_1}^{\frac{2^{\ast}-1}{2\alpha-1}}\right)\nonumber\\
&\leq C\left(\left(\theta+\|v\|_{p_1,p_2,R}\right)^{\frac{2^{\ast}-1}{2\alpha-1}}+1\right),
\end{align}
and
\begin{align}\label{a18}
\|w_2\|_{q_2}&\leq C\|v_2\|_{p_2}^{\frac{2^{\ast}-1}{2\alpha-1}}
\leq C\left(R^{\frac{N}{2^{\ast}}-\frac{N}{p_2}}\left(\theta+\|v_2\|_{p_1,p_2,R}\right)\right)^{\frac{2^{\ast}-1}{2\alpha-1}},\nonumber\\
&=CR^{\frac{N}{2^{\ast}}-\frac{N}{q_2}}\left(\theta+\|v\|_{p_1,p_2,R}\right)^{\frac{2^{\ast}-1}{2\alpha-1}},
\end{align}
since $$\left(\frac{N}{2^{\ast}}-\frac{N}{p_2}\right)\frac{2^{\ast}-1}{2\alpha-1}-\left(\frac{N}{2^{\ast}}-\frac{N}{q_2}\right)
=N\left(\frac{2\alpha}{N(2\alpha-1)}-\frac{2^{\ast}-1}{p_2(2\alpha-1)}+\frac{1}{q_2}\right)=0.$$
So, the result follows.
\end{proof}

Let $u_{\alpha}$ be a solution of \eqref{P}, and $w_{\alpha}$ be the solution of
\begin{equation} \label{a19}
\left\{
\begin{aligned} 
 &-\mbox{div}((1+|\nabla w|^2)^{\alpha-1}\nabla w)=2u_{\alpha}^{2^{\ast}-1}+A\ \ \ \mbox{in}\ \Omega,\\
&w=0,\ \  \mbox{on}\ \partial\Omega,
\end{aligned}
\right.
\end{equation}
where $A > 0$ is a large constant. By the comparison principle, we have
\begin{equation}\label{a20}
|u_{\alpha}(x)|\leq w_{\alpha} (x), x\in \Omega.
\end{equation}

\begin{lemma}\label{l35}
Let $w_{\alpha}(x)$ be a solution of (\ref{a19}). Then there exist  constants $C>0$, and $p_1, p_2\in \left(\left(1-\frac{1}{2\alpha}\right)
2^{\ast},+\infty\right)$ with $p_2<2^{\ast}<p_1$, such that
$$\|w_{\alpha}\|_{p_1,p_2,R_{\alpha}}\leq C.$$
\end{lemma}
\begin{proof}
From Theorem \ref{p21}, we have
  \begin{equation}\label{a21}
   u_\alpha=u+\sum_{j=1}^k\rho_{x_{\alpha,j},R_{\alpha,j}}(U_j)+\omega_\alpha,
  \end{equation}
where $\rho_{x_{\alpha,j},R_{\alpha,j}}(U_j)=(R_{\alpha}^j)^{\frac{N-2}{2}}U_{j}\left(R_{\alpha}^j(x-x_{\alpha}^j)\right)$,
 $u_0$ is a $C^{1,\alpha}$ function, and 	$\|\omega_\alpha\|\rightarrow 0$ as $\alpha\rightarrow1$.
Let $\tilde{w}_{\alpha}$ be a solution of
\begin{equation} \label{a22}
\left\{
\begin{aligned} 
 &-\mbox{div}((1+|\nabla \tilde{w}_{\alpha}|^2)^{\alpha-1}\nabla \tilde{w}_{\alpha})
 =C\left(|u|^{2^{\ast}-2\alpha}
 +\sum_{j=1}^k|\rho_{x_{\alpha,j},R_{\alpha,j}}(U_j)|^{2^{\ast}-2\alpha}+|\omega_\alpha|^{2^{\ast}-2\alpha}\right)w_{\alpha}^{2\alpha-1}+A\  \mbox{in}\ \Omega_1,\\
&\tilde{w}_{\alpha}=0,\ \  \mbox{on}\ \partial\Omega_1,
\end{aligned}
\right.
\end{equation}
where $\Omega_1$ is a bounded domain in $\mathbb{R}^N$ satisfying $\Omega \subset \Omega_1, C>0$ is a fixed large constant. Then,
by the comparison principle,
$$w_{\alpha}(x)\leq \tilde{w}_{\alpha}(x)\ \ x \in \Omega.$$
Moreover, it is easy to check that
\begin{equation}\label{a24}
\int\limits_{\Omega_1}|\tilde{w}_{\alpha}(x)|^{2^{\ast}}dx\leq C.
\end{equation}

Let $w = Gv$ be the solution of
\begin{equation} \label{a23}
\left\{
\begin{aligned} 
 &-\mbox{div}((1+|\nabla w|^2)^{\alpha-1}\nabla w)=v\ \ \ \mbox{in}\ \Omega_1,\\
&w=0,\ \  \mbox{on}\ \partial\Omega_1.
\end{aligned}
\right.
\end{equation}
Then, it follows from Lemma \ref{l31} and (\ref{a24}) that
$$\tilde{w}_\alpha\leq C+G\left(|u|^{2^{\ast}-2\alpha}w_{\alpha}^{2\alpha-1}+A\right)
 +\sum_{j=1}^kG\left(|\rho_{x_{\alpha,j},R_{\alpha,j}}(U_j)|^{2^{\ast}-2\alpha}w_{\alpha}^{2\alpha-1}\right)
 +G\left(|\omega_\alpha|^{2^{\ast}-2\alpha}w_{\alpha}^{2\alpha-1}\right).$$

Firstly, we treat the term $G\left(|u|^{2^{\ast}-2\alpha}w_{\alpha}^{2\alpha-1}+A\right)$. Take any $q\in \left(\frac{2\alpha N}{2\alpha N-N+2\alpha},\frac{2^{\ast}}{2\alpha-1}\right)$. Then
$$p_1=\frac{Nq(2\alpha-1)}{N-2\alpha q}>(2\alpha)^{\ast}>2^{\ast}.$$
It follows from Proposition \ref{app1} that
\begin{align}\label{a25}
\|G\left(|u|^{2^{\ast}-2\alpha}w_{\alpha}^{2\alpha-1}+A\right)\|_{p_1}&\leq C \||u|^{2^{\ast}-2\alpha}w_{\alpha}^{2\alpha-1}+A\|_{q}^{\frac{1}{2\alpha-1}}\nonumber\\
&\leq C+C\left(\int\limits_{\Omega}\left||u|^{2^{\ast}-2\alpha}w_{\alpha}^{2\alpha-1}\right|^qdx\right)^{\frac{1}{q(2\alpha-1)}}\nonumber\\
&\leq C+C\left(\int\limits_{\Omega}\left|w_{\alpha}\right|^{(2\alpha-1) q}dx\right)^{\frac{1}{q(2\alpha-1)}}\leq C.
\end{align}

Next, we treat the term $G\left(|\rho_{x_{\alpha,j},R_{\alpha,j}}(U_j)|^{2^{\ast}-2\alpha}w_{\alpha}^{2\alpha-1}\right)$.
Let $p_2\in \left(\frac{(2\alpha-1)N}{N-2\alpha}, 2^{\ast}\right)$ be a constant.
By Corollary \ref{app13}, we obtain
\begin{align}\label{a26}
\|G\left(|\rho_{x_{\alpha,j},R_{\alpha,j}}(U_j)|^{2^{\ast}-2\alpha}w_{\alpha}^{2\alpha-1}\right)\|_{p_2}
&\leq C \left\||\rho_{x_{\alpha,j},R_{\alpha,j}}(U_j)|^{2^{\ast}-2\alpha}\right\|_r^{\frac{1}{2\alpha-1}} \|w_{\alpha}\|_{2^{\ast}} \nonumber\\
&\leq C \left\||\rho_{x_{\alpha,j},R_{\alpha,j}}(U_j)|^{2^{\ast}-2\alpha}\right\|_r^{\frac{1}{2\alpha-1}},
\end{align}
where $r$ is determined by $$\frac{1}{r}=\frac{2\alpha-1}{p_2}+\frac{2\alpha}{N}-\frac{2\alpha-1}{2^{\ast}}.$$
But $$ \int\limits_{\Omega} |\rho_{x_{\alpha,j},R_{\alpha,j}}(U_j)|^{(2^{\ast}-2\alpha)r}dx
=(R_{\alpha,j})^{-N+(N-\alpha N+2\alpha)r}\int\limits_{\Omega_{x_{\alpha,j},R_{\alpha,j}}}|U_j|^{\frac{2N-2\alpha N+4\alpha}{N-2}r}dx, $$
where $\Omega_{x,\lambda}=\{y: x+\lambda^{-1}y\in \Omega\}$.

For $j = 1,\cdots, k$, using  the Kelvin transformation $v(x)=|x|^{2-N}U\left(\frac{x}{|x|^2}\right)$, we have
\begin{equation}\label{a27}
|U_j(x)|\leq \frac{C}{|x|^{N-2}}.
\end{equation}
So for any $r>\frac{N}{4\alpha-(2\alpha-2)N}$, we obtain
$$\int\limits_{\Omega_{x_{\alpha,j},R_{\alpha,j}}}|U_j|^{\frac{2N-2\alpha N+4\alpha}{N-2}r}dx\leq C,\ j=1,\cdots,k.$$
Note that $\frac{N}{2\alpha}>\frac{N}{4\alpha-(2\alpha-2)N}$  and $r\rightarrow \frac{N}{2\alpha}$ as $p_2\rightarrow 2^{\ast}$. So we
such that the corresponding $r>\frac{N}{4\alpha-(2\alpha-2)N}$. Thus, we have proved that there is a $p_2 <2^{\ast}$, such that
\begin{equation}\label{a29}
\|G\left(|\rho_{x_{\alpha,j},R_{\alpha,j}}(U_j)|^{2^{\ast}-2\alpha}w_{\alpha}^{2\alpha-1}\right)\|_{p_2}
\leq C(R_{\alpha,j})^{\left(\frac{-N}{r}+(N-\alpha N+2\alpha)\right)\frac{1}{2\alpha-1}}
=C(R_{\alpha,j})^{\frac{N}{2^{\ast}}-\frac{N}{p_2}+\frac{N-\alpha N}{2\alpha-1}}.
\end{equation}

Finally, we treat the term $G\left(|\omega_\alpha|^{2^{\ast}-2\alpha}w_{\alpha}^{2\alpha-1}\right)$. It follows from Lemma \ref{l32} that
\begin{align}\label{a28}
 \|G\left(|\omega_\alpha|^{2^{\ast}-2\alpha}w_{\alpha}^{2\alpha-1}\right)\|_{p_1,p_2.R_{\alpha}} &\leq C\||\omega_{\alpha}|^{2^{\ast}-2\alpha}\|^{\frac{1}{2\alpha-1}}_{\frac{N}{2\alpha}}\|w_{\alpha}\|_{p_1,p_2,R_{\alpha}}\nonumber\\
 &\leq \frac{1}{2}\|w_n\|_{p_1,p_2, R_{\alpha}}+C.
\end{align}
Since $\||\omega_{\alpha}|^{2^{\ast}-2\alpha}\|_{\frac{N}{2\alpha}}\rightarrow 0$  as $\alpha\rightarrow 1$.

Combining (\ref{a26}), (\ref{a29}) and (\ref{a28}), we obtain
\begin{align}\label{a30}
   &\|w_{\alpha}\|_{p_1,p_2,R_{\alpha}}  \nonumber\\
   & \leq C+\|G\left(|u|^{2^{\ast}-2\alpha}w_{\alpha}^{2\alpha-1}+A\right)\|_{p_1,p_2,R_{\alpha}}
 +\sum_{j=1}^k\|G\left(|\rho_{x_{\alpha,j},R_{\alpha,j}}(U_j)|^{2^{\ast}-2\alpha}w_{\alpha}^{2\alpha-1}\right)\|_{p_1,p_2,R_{\alpha}}\nonumber\\
&\ \ \  +\|G\left(|\omega_\alpha|^{2^{\ast}-2\alpha}w_{\alpha}^{2\alpha-1}\right)\|_{p_1,p_2,R_{\alpha}}.\nonumber\\
   & \leq C+\|G\left(|u|^{2^{\ast}-2\alpha}w_{\alpha}^{2\alpha-1}+A\right)\|_{p_1}
 +C\sum_{j=1}^k\|G\left(|\rho_{x_{\alpha,j},R_{\alpha,j}}(U_j)|^{2^{\ast}-2\alpha}w_{\alpha}^{2\alpha-1}\right)\|
 _{p_2}(R_{\alpha,j})^{\frac{N}{2^{\ast}}-\frac{N}{p_2}+\frac{N-\alpha N}{2\alpha-1}}\nonumber\\
&\ \ \  +C\|G\left(|\omega_\alpha|^{2^{\ast}-2\alpha}w_{\alpha}^{2\alpha-1}\right)\|_{p_1,p_2,R_{\alpha}}.\nonumber\\
 &\leq \frac{1}{2}\|w_{\alpha}\|_{p_1,p_2, R_{\alpha}}+C.
\end{align}
So the result follows.

\end{proof}

\begin{proof}[\textbf{Proof of Proposition \ref{p31}}]
The constants $q_1$ and $q_2$ defined in Lemma \ref{l34} satisfy  $q_1 > 2^{\ast}$
 and
$q_1\rightarrow +\infty $ as $p_1\rightarrow  \frac{N}{2\alpha}(2^{\ast}-1)$, while $q_2 <2^{\ast}$
and $q_2\rightarrow(1 - \frac{1}{2\alpha})2^{\ast}$ as $p_2\rightarrow 2^{\ast}-1$.

Using (\ref{a20}), we just need to show the result for $w_n$.

Since $w_n$ satisfies (\ref{a19}), we can use Lemmas \ref{l34} and \ref{l35} to prove that
	$$\|w_{\alpha}\|_{p_1,p_2,R_{\alpha}} \leq  C$$
holds for any $p_1,p_2$ with $p_1 >2^{\ast}>p_2 >\left(1-\frac{1}{2\alpha}\right)2^{\ast}$,
\end{proof}

\subsection{Estimates on safe regions}
Noting that the number of the bubbles of un is finite and using Theorem \ref{p21}, we can always
find a constant  $\bar{C} >0$, independent of $n$, such that the region

$$\mathcal{A}_{\alpha}^1=\left(B_{(\bar{C}+5)R_{\alpha}^{-\frac{1}{2}}}(x_{\alpha})\setminus B_{\bar{C}R_{\alpha}^{-\frac{1}{2}}}(x_{\alpha}) \right)\cap \Omega$$
does not contain any concentration point of $u_{\alpha}$ for any $\alpha$. We call this region a safe region for $u_{\alpha}$.

Let
$$\mathcal{A}_{\alpha}^2=\left(B_{(\bar{C}+4)R_{\alpha}^{-\frac{1}{2}}}(x_{\alpha})\setminus B_{(\bar{C}+1)R_{\alpha}^{-\frac{1}{2}}}(x_{\alpha}) \right)\cap \Omega.$$

In this section, we will prove the following result.

\begin{proposition}\label{p36}
Let $u_{\alpha}$ be a weak solution of (\ref{P}) with $\alpha\rightarrow 1$.
Then there exists  a constant $C>0$, independent of $\alpha$, such that
$$\int\limits_{\mathcal{A}_{\alpha}^2}|u_{\alpha}(x)|dx\leq CR_{\alpha}^{\frac{-N}{2}}, \ \mbox{for\ all} \ x\in \mathcal{A}_{\alpha}^2.$$
\end{proposition}

To prove Proposition \ref{p36}, we need the following lemma.
\begin{lemma}\label{l37}
There is a constant $C>0$, independent of $\alpha$, such that
$$\frac{1}{r^N}\int\limits_{\partial B_r(y)\cap \Omega} |u_{\alpha}| dx \leq C,\ \ \mbox{for\ all}\ y\in\Omega,$$
for all $r\geq \bar{C}R_{\alpha}^{-\frac{1}{2}}$.
\end{lemma}
\begin{proof}
Let $\tilde{w}_{\alpha}$ be the solution of

\begin{equation} \label{a31}
\left\{
\begin{aligned} 
 &-\Delta \tilde{w}_{\alpha}=2|u_{\alpha}|^{2^{\ast}-1}+A\ \ \ \mbox{in}\ \Omega_1,\\
&\tilde{w}_{\alpha}=0,\ \  \mbox{on}\ \partial\Omega_1,
\end{aligned}
\right.
\end{equation}
and we konw that $\tilde{w}_{\alpha}>0$.
Then we have
\begin{align}\label{b32}
  0&<\mbox{div}((1+|\nabla u_{\alpha}|^2)^{\alpha-1}\nabla u_{\alpha})-\Delta \tilde{w}_{\alpha}\nonumber\\
  & =\mbox{div}((1+|\nabla u_{\alpha}|^2)^{\alpha-1}\nabla u_{\alpha})-\mbox{div}((1+|\nabla \tilde{w}_{\alpha}|^2)^{\alpha-1}\nabla \tilde{w}_{\alpha})\nonumber\\
  &\ \ + \mbox{div}((1+|\nabla \tilde{w}_{\alpha}|^2)^{\alpha-1}\nabla \tilde{w}_{\alpha})-\Delta \tilde{w}_{\alpha}.\nonumber\\
  & =\mbox{div}((1+|\nabla u_{\alpha}|^2)^{\alpha-1}\nabla u_{\alpha})-\mbox{div}((1+|\nabla \tilde{w}_{\alpha}|^2)^{\alpha-1}\nabla \tilde{w}_{\alpha})\nonumber\\
  &\ \ + \mbox{div}((1+|\nabla \tilde{w}_{\alpha}|^2)^{\alpha-1}-1)\nabla \tilde{w}_{\alpha}).
\end{align}

If $\mbox{div}((1+|\nabla \tilde{w}_{\alpha}|^2)^{\alpha-1}-1)-\nabla \tilde{w}_{\alpha})\leq 0$, using \eqref{b32}   and
the comparison principle of uniform elliptic operator,
 we have $|u_{\alpha}|\leq  \tilde{w}_{\alpha}$ in $\Omega$.

If  $\mbox{div}((1+|\nabla \tilde{w}_{\alpha}|^2)^{\alpha-1}-1)\nabla \tilde{w}_{\alpha})>0$,  using the
by the comparison principle in \cite{fn}, we have $\tilde{w}_{\alpha}<0$, which is a contradiction.

Hence, we have  $|u_{\alpha}|\leq  \tilde{w}_{\alpha}$ in $\Omega$. 

Now we have the following formula
\begin{align}\label{a32}
\frac{1}{r^{N-1}}\int\limits_{\partial B_r(y)}\tilde{w}_{\alpha}dS&=\frac{1}{r_0^{N-1}}\int\limits_{\partial B_{r_0}(y)}\tilde{w}_{\alpha}dS+\int\limits_r^{r_0} \frac{1}{t^{N-1}}\int\limits_{B_t(y)}(2|u_{\alpha}|^{2^{\ast}-1}+A)dxdt.
\end{align}

By continuity, since $\{u_{\alpha}\}$ is  bounded in $L^{2^{\ast}}\subset L^1$, we can suppose
$$\int_{B_1(y)} u_{\alpha}dx\leq C$$ with a constat $C$ independent of $\alpha$. So, there is
$r_0\in [\frac{1}{2}, 1]$, such that
  $$\frac{1}{r_0^{N-1}}\int_{\partial B_{r_0}(y)} u_{\alpha}dx\leq C.$$

We now estimate
$$\int\limits_r^{r_0} \frac{1}{t^{N-1}}\int\limits_{B_t(y)}|u_{\alpha}|^{2^{\ast}-1}dxdt,\
\ \mbox{for\ all\ } r\geq \bar{C}R_{\alpha}^{-\frac{1}{2}}.$$
By Proposition \ref{p31}, we know that 	$\|u_{\alpha}\|_{p_1,p_2,R_{\alpha}} \leq C$  for any $p_1, p_2$ satisfing
$p_1 >2^{\ast}>p_2 >\left(1-\frac{1}{2\alpha}\right)2^{\ast}$.

Let $p_2 = 2^{\ast} - 1$, and let $p_1 >2^{\ast}$. Then we can choose $v_{1,\alpha}$ and $v_{2,\alpha}$, such that
 $|u_{\alpha}| \leq v_{1,\alpha} +v_{2,\alpha}$, and

$$\|v_{1,\alpha}\|_{p_1}\leq C,\ \ \ \|v_{2,\alpha}\|_{p_2}\leq C R^{\frac{N}{2^{\ast}}-\frac{N}{p_2}}.$$

If we choose $p_1 >2^{\ast}$ large enough, then

\begin{align}\label{a33}
\int\limits_r^{r_0}\frac{1}{t^{N-1}}\int\limits_{B_t(y)}v_{1,\alpha}^{2^{\ast}-1}dxdt
&\leq \int\limits_r^{r_0}\frac{1}{t^{N-1}}\left(\int\limits_{B_t(y)}v_{1,\alpha}^{p_1}dx\right)^{\frac{2^{\ast}-1}{p_1}}
t^{N\left(1-\frac{2^{\ast}-1}{p_1}\right)}dt,\nonumber\\
&\leq C\int\limits_r^{r_0} \frac{1}{t^{N-1}}t^{N\left(1-\frac{2^{\ast}-1}{p_1}\right)}dt\leq C,
\end{align}

and

\begin{align}\label{a34}
\int\limits_r^{r_0}\frac{1}{t^{N-1}}\int\limits_{B_t(y)}v_{2,\alpha}^{2^{\ast}-1}dxdt
&\leq C\int\limits_r^{r_0} \frac{1}{t^{N-1}}R_{\alpha}^{(2^{\ast}-1)\left(\frac{N}{2^{\ast}}-\frac{N}{2^{\ast}-1}\right)}dt,\nonumber\\
&\leq C\frac{1}{r^{N-2}}R_{\alpha}^{-\frac{N-2}{2}}\leq C,
\end{align}
since $r\geq \bar{C}R_{\alpha}^{-\frac{1}{2}}.$

Combining (\ref{a33}) and (\ref{a34}), we obtain
\begin{align}\label{a35}
\int\limits_r^{r_0} \left(\frac{1}{t^{N-2}}\int\limits_{B_t(y)}|u_{\alpha}|^{2^{\ast}-1}dx\right)
\frac{dt}{t}&\leq  C\int\limits_r^{r_0}\frac{1}{t^{N-1}}\int\limits_{B_t(y)}v_{1,\alpha}^{2^{\ast}-1}dxdt \nonumber\\
& \ \ +C\int\limits_r^{r_0}\frac{1}{t^{N-1}}\int\limits_{B_t(y)}v_{2,\alpha}^{2^{\ast}-1}dxdt\leq C.
\end{align}

\end{proof}

The conclusion for $r \geq  r_0$ is obvious and thus we complete our proof.

Now we are ready to prove Proposition \ref{p36}.

\begin{proof}[\textbf{Proof of Proposition \ref{p36}}]
It follows from Lemma  \ref{l37}
that for any $y \in \mathcal{A}_{\alpha}^2$, we have
$$\frac{1}{r^N}\int\limits_{\partial B_r(y)\cap \Omega} |u_{\alpha}| dx \leq C,\ \ \mbox{for\ all}\ y\in\Omega,$$
for all $r\geq \bar{C}R_{\alpha}^{-\frac{1}{2}}$.
Let $$v_{\alpha}(x)=u_{\alpha}\left(R_{\alpha}^{-\frac{1}{2}}x\right),\ \ x\in \Omega_{\alpha},$$
where $\Omega_{\alpha}=\{x: R_{\alpha}^{-\frac{1}{2}}x \in \Omega\}$.

Then from \eqref{app17}, $v_{\alpha}$ satisfies
\begin{align}\label{a36}
\int\limits_{\Omega_{\alpha}} |\nabla v_{\alpha}|^2 ]dx
&\leq \int\limits_{\Omega_{\alpha}} [(1+|\nabla v_{\alpha}|^2)^{\alpha-1}|\nabla v_{\alpha}|^2 ]dx\nonumber\\
&\leq R_{\alpha}^{\frac{N}{2}}\int\limits_{\Omega} [(1+|R_{\alpha}^{-\frac{1}{2}}\nabla u_{\alpha}|^2)^{\alpha-1}|R_{\alpha}
^{-\frac{1}{2}}\nabla u_{\alpha}|^2 ]dx\nonumber\\
&\leq R_{\alpha}^{\frac{N}{2}}R_{\alpha}^{-\frac{1}{2}} \int\limits_{\Omega} [(1+|\nabla u_{\alpha}|^2)^{\alpha-1}|\nabla u_{\alpha}|^2 ]dx\nonumber\\
&\leq  CR_{\alpha}^{\frac{N}{2}} R_{\alpha}
^{-1}\int\limits_{\Omega}(2|u_{\alpha}|^{2^{\ast}}+Au_{\alpha}^2)dx\nonumber\\
& \leq C R_{\alpha}
^{-1}\int\limits_{\Omega_{\alpha}}(2|v_{\alpha}|^{2^{\ast}-1}+A|v_{\alpha}|)|v_{\alpha}|dx.
\end{align}

Let $z=R_{\alpha}^{\frac{1}{2}}y$. Since $B_{R_{\alpha}^{-\frac{1}{2}}}(y)$, $y\in \mathcal{A}_{\alpha}^2$
does not contain any concentration point of $u_\alpha$, we can
deduce that
\begin{align}\label{a37}
   R_{\alpha}^{-1}\int\limits_{B_1(z)}(2|v_{\alpha}|^{2^{\ast}-1}+A|v_{\alpha}|)^{\frac{N}{2}}dx\leq \int\limits_{B_{R^{-\frac{1}{2}}_{\alpha}}}
   |u_{\alpha}(x)|^{2^{\ast}}dx+CR_{\alpha}^{-\frac{N}{2}}\rightarrow 0, \ \ \alpha\rightarrow 1.
\end{align}

Since $v_{\alpha}$ satisfies \eqref{a36},  by Moser iteration,  we obtain
\begin{align}\label{a38}
\|v_{\alpha}\|_{L^{\infty}(B_{\frac{1}{2}}(z))}&\leq C\int\limits_{B_1(z)} |v_{\alpha}|dx +C\nonumber\\
&=CR_{\alpha}^{\frac{N}{2}}\int\limits_{B_{R^{-\frac{1}{2}}_{\alpha}}(y)} |u_{\alpha}|dx+C\leq C.
\end{align}
As a result,
\begin{equation}\label{a39}
  \int\limits_{B_{\frac{1}{2}R^{-\frac{1}{2}}_{\alpha}}(y)} |u_{\alpha}|dx \leq CR_{\alpha}^{-\frac{N}{2}}, \mbox{\ for\ all\ } y \in \mathcal{A}_{\alpha}^2.
\end{equation}
\end{proof}

Let
$$\mathcal{A}_{\alpha}^3=\left(B_{(\bar{C}+3)R_{\alpha}^{-\frac{1}{2}}}(x_{\alpha})
\setminus B_{(\bar{C}+2)R_{\alpha}^{-\frac{1}{2}}}(x_{\alpha}) \right)\cap \Omega.$$

\begin{proposition}\label{p37}
\begin{equation}\label{a42}
\int\limits_{\mathcal{A}_{\alpha}^3} |\nabla u_{\alpha}|^{2\alpha}dx
 \leq C\int\limits_{\mathcal{A}_{\alpha}^2}  (|u_{\alpha}|^{2^{\ast}}+1)dx
+CR_{\alpha}\int\limits_{\mathcal{A}_{\alpha}^2}  |u_{\alpha}|^{2}dx.
\end{equation}
In particular,
\begin{equation}\label{a43}
\int\limits_{\mathcal{A}_{\alpha}^3} |\nabla u_{\alpha}|^{2\alpha}dx \leq CR_{\alpha}^{\frac{2-N}{2}}.
\end{equation}
\end{proposition}
\begin{proof}
Let $\phi_{\alpha}\in  C_0^2 (\mathcal{A}_{\alpha}^2)$ be a function with $\phi_{\alpha} = 1$ in $\mathcal{A}_{\alpha}^3$,
$0\leq \phi_{\alpha}\leq 1$ and $|\nabla \phi_{\alpha}|\leq CR_{\alpha}^{\frac{1}{2}}$.
From
\begin{equation}\label{a45}
\int\limits_{\Omega} (1+|\nabla u_{\alpha}|^{2})^{\alpha-1}\nabla u_{\alpha}\nabla (\phi_{\alpha}^2 u_{\alpha})dx
\leq C\int\limits_{\Omega}  (2|u_{\alpha}|^{2^{\ast}-2}+A)\phi_{\alpha}^2|u_{\alpha}^2|dx,
\end{equation}
we can prove \eqref{a45}.

From \eqref{a45} and Proposition \eqref{p36}, we get
\begin{equation}\label{a46}
\int\limits_{\mathcal{A}_{\alpha}^3} |\nabla u_{\alpha}|^{2\alpha}dx \leq CR_{\alpha}^{\frac{-N}{2}}+CR_{\alpha}^{\frac{2-N}{2}}
\leq CR_{\alpha}^{\frac{2-N}{2}}.
\end{equation}
\end{proof}

\subsection{Proof of Theorem \ref{rrr}}
Take a $t_n \in  [\bar{C} +2, \bar{C} +3]$, satisfying
\begin{align}\label{a47}
&\int\limits_{\partial B_{t_{\alpha}R^{-\frac{1}{2}}_{\alpha}}(y)} (R_{\alpha}^{-1}|u_{\alpha}|^{2^{\ast}}
+|u_{\alpha}|^2+R_{\alpha}^{-1}|\nabla u_{\alpha}|^{2\alpha})dS\nonumber \\
& \leq CR_{\alpha}^{\frac{1}{2}}\int\limits_{\mathcal{A}_{\alpha}^3} (R_{\alpha}^{-1}|u_{\alpha}|^{2^{\ast}}
+|u_{\alpha}|^2+R_{\alpha}^{-1}|\nabla u_{\alpha}|^{2\alpha})dS.
\end{align}
Using Proposition \ref{p36}, (\ref{a43}) and (\ref{a47}), we obtain
\begin{align}\label{a48}
&\int\limits_{\partial B_{t_{\alpha}R^{-\frac{1}{2}}_{\alpha}}(y)} (R_{\alpha}^{-1}|u_{\alpha}|^{2^{\ast}}
+|u_{\alpha}|^2+R_{\alpha}^{-1}|\nabla u_{\alpha}|^{2\alpha})dS
\leq CR_{\alpha}^{\frac{1}{2}-\frac{N}{2}}.
\end{align}

\begin{proof}[\textbf{Proof of Theorem \ref{rrr}}]
We have two different cases:
\begin{itemize}\addtolength{\itemsep}{-1.5 em} \setlength{\itemsep}{-5pt}
\item[(i)]  $B_{t_{\alpha}R^{-\frac{1}{2}}_{\alpha}}(x_{\alpha})\cap (\mathbb{R}^N\setminus \Omega)\not=\emptyset$;

\item[(ii)]  $B_{t_{\alpha}R^{-\frac{1}{2}}_{\alpha}}(x_{\alpha})\subset \Omega.$
\end{itemize}
 We have the following local Pohozaev identity for $u_{\alpha}$ on $B_{\alpha}=B_{t_{\alpha}R^{-\frac{1}{2}}_{\alpha}}(x_{\alpha})\cap \Omega$:
\begin{align}\label{a44}
  &\lambda \int\limits_{B_{\alpha}} u_{\alpha}^2dx+  \frac{N}{2}\int\limits_{B_{\alpha}} |\nabla u_{\alpha} |^2(1+|\nabla u_{\alpha}|^{\alpha-1})dx
 - \frac{N}{2\alpha}\int\limits_{B_{\alpha}} [(1+|\nabla u_{\alpha}|^2)^{\alpha}-1]dx  \nonumber \\
   & = \frac{1}{2^{\ast}}\int\limits_{\partial B_{\alpha}} |u_{\alpha}|^{2^{\ast}}(x-x_0)\cdot \nu dS
   +  \frac{\lambda}{2}\int\limits_{\partial B_{\alpha}} |u_{\alpha}|^{2}(x-x_0)\cdot \nu dS \nonumber \\
    &\ \ -\frac{1}{2\alpha}\int\limits_{\partial B_{\alpha}}[( 1+|\nabla u_{\alpha}|^{2})^{\alpha}-1] (x-x_0))\cdot  \nu dS
   +\frac{N}{2^{\ast}}\int\limits_{\partial B_{\alpha}}( 1+|\nabla u_{\alpha}|^{2})^{\alpha -1}  (\nabla u_{\alpha}\cdot \nu)u_{\alpha}dS\nonumber \\
    &\ \ +\int\limits_{\partial B_{\alpha}}( 1+|\nabla u_{\alpha}|^{2})^{\alpha -1} (\nabla u_{\alpha}\cdot(x-x_0))\cdot  (\nabla u_{\alpha}\cdot \nu) dS,
\end{align}
where $\nu$ is the outward normal to $\partial B_{\alpha}$. The point $x_0$ in (\ref{a44}) is chosen as follows. In case (i), we
take $x_0 \in  \mathbb{R}^N\setminus \Omega $ with $|x_0 -x_{\alpha}|\leq  2t_{\alpha} R^{-\frac{1}{2}}_{\alpha}$
 and $\nu \cdot(x -x_0) \leq 0 $ in $\partial\Omega \cap B_{\alpha}$. In case (ii), we take
a point $x_0= x_{\alpha}$.

By \eqref{app17}, we know that the first term in the left-hand side of (\ref{a44})  is non-negative. We thus obtain
from (\ref{a44}) that

\begin{align}\label{a49}
  \lambda \int\limits_{B_{\alpha}} u_{\alpha}^2dx&
  \leq  \frac{1}{2^{\ast}}\int\limits_{\partial B_{\alpha}} |u_{\alpha}|^{2^{\ast}}(x-x_0)\cdot \nu dS
   +  \frac{\lambda}{2}\int\limits_{\partial B_{\alpha}} |u_{\alpha}|^{2}(x-x_0)\cdot \nu dS \nonumber \\
    &\ \ -\frac{1}{2\alpha}\int\limits_{\partial B_{\alpha}}[( 1+|\nabla u_{\alpha}|^{2})^{\alpha}-1] (x-x_0))\cdot  \nu dS
   +\frac{N}{2^{\ast}}\int\limits_{\partial B_{\alpha}}( 1+|\nabla u_{\alpha}|^{2})^{\alpha -1}  (\nabla u_{\alpha}\cdot \nu)u_{\alpha}dS\nonumber \\
    &\ \ +\int\limits_{\partial B_{\alpha}}( 1+|\nabla u_{\alpha}|^{2})^{\alpha -1} (\nabla u_{\alpha}\cdot(x-x_0))\cdot  (\nabla u_{\alpha}\cdot \nu) dS.
\end{align}
Now we decompose $\partial B_{\alpha}$ into
$$\partial B_{\alpha}=\partial_i B_{\alpha}\cup \partial_e B_{\alpha},$$
where $\partial_i B_{\alpha}=\partial B_{\alpha}\cap \partial \Omega$ and $\partial_e B_{\alpha}=\partial B_{\alpha}\cap \partial\Omega$.

Noting $u_{\alpha}=0$ on $\partial\Omega$, we find
\begin{align}\label{a50}
&\frac{1}{2^{\ast}}\int\limits_{\partial_e B_{\alpha}} |u_{\alpha}|^{2^{\ast}}(x-x_0)\cdot \nu dS
   +  \frac{\lambda}{2}\int\limits_{\partial_e B_{\alpha}} |u_{\alpha}|^{2}(x-x_0)\cdot \nu dS \nonumber \\
    &\ \ -\frac{1}{2\alpha}\int\limits_{\partial_e B_{\alpha}}[( 1+|\nabla u_{\alpha}|^{2})^{\alpha}-1] (x-x_0))\cdot  \nu dS
   +\frac{N}{2^{\ast}}\int\limits_{\partial_e B_{\alpha}}( 1+|\nabla u_{\alpha}|^{2})^{\alpha -1}  (\nabla u_{\alpha}\cdot \nu)u_{\alpha}dS\nonumber \\
    &\ \ +\int\limits_{\partial_e B_{\alpha}}( 1+|\nabla u_{\alpha}|^{2})^{\alpha -1} (\nabla u_{\alpha}\cdot(x-x_0))\cdot  (\nabla u_{\alpha}\cdot \nu) dS\nonumber \\
    &=  -\frac{1}{2\alpha}\int\limits_{\partial_e B_{\alpha}}[( 1+|\nabla u_{\alpha}|^{2})^{\alpha}-1] (x-x_0))\cdot  \nu dS \nonumber \\
    &\ \ +\int\limits_{\partial_e B_{\alpha}}( 1+|\nabla u_{\alpha}|^{2})^{\alpha -1} (\nabla u_{\alpha}\cdot(x-x_0))\cdot  (\nabla u_{\alpha}\cdot \nu) dS\leq 0.
\end{align}
So, we can rewrite (\ref{a49}) as
\begin{align}\label{a51}
  \lambda \int\limits_{B_{\alpha}} u_{\alpha}^2dx&
  \leq  \frac{1}{2^{\ast}}\int\limits_{\partial_i B_{\alpha}} |u_{\alpha}|^{2^{\ast}}(x-x_0)\cdot \nu dS
   +  \frac{\lambda}{2}\int\limits_{\partial_i B_{\alpha}} |u_{\alpha}|^{2}(x-x_0)\cdot \nu dS \nonumber \\
    &\ \ -\frac{1}{2\alpha}\int\limits_{\partial_i B_{\alpha}}[( 1+|\nabla u_{\alpha}|^{2})^{\alpha}-1] (x-x_0))\cdot  \nu dS
   +\frac{N}{2^{\ast}}\int\limits_{\partial_i B_{\alpha}}( 1+|\nabla u_{\alpha}|^{2})^{\alpha -1}  (\nabla u_{\alpha}\cdot \nu)u_{\alpha}dS\nonumber \\
    &\ \ +\int\limits_{\partial_i B_{\alpha}}( 1+|\nabla u_{\alpha}|^{2})^{\alpha -1} (\nabla u_{\alpha}\cdot(x-x_0))\cdot  (\nabla u_{\alpha}\cdot \nu) dS,
  \   (\mbox{by}\  \eqref{app17}).
\end{align}

Using (\ref{a48}), noting that $|x_0 -x_{\alpha}|\leq  C R^{-\frac{1}{2}}_{\alpha}$ for
$\partial_i B_{\alpha}$, we see
\begin{align}\label{a52}
  RHS of (\ref{a51}) &\leq  CR^{-\frac{1}{2}}_{\alpha}\int\limits_{\partial_i B_{\alpha}}
  (|u_{\alpha}|^{2^{\ast}}+|u_{\alpha}|^{2}+|\nabla u_{\alpha}|^{2})dS
  + C\int\limits_{\partial_i B_{\alpha}} |\nabla u_{\alpha}||u_{\alpha}|dS\nonumber\\
  &\leq CR^{-\frac{N-2}{2}}_{\alpha}.
\end{align}
Recall that in the proof of Lemma \ref{l35}, we have the decomposition
  \begin{equation}\label{a53}
   u_\alpha=u+\sum_{j=1}^k\rho_{x_{\alpha,j},R_{\alpha,j}}(U_j)+\omega_\alpha:=u_0+u_{\alpha,1}+u_{\alpha,2},
  \end{equation}
with 	$\|u_{\alpha,2}\|\rightarrow 0$ as $\alpha\rightarrow 1$.  We easily find that if $N >4$,
\begin{equation}\label{a54.5}
  \int\limits_{\mathbb{R}^N} |U_j|^2dx<+\infty,\ j=1,\cdots,k.
\end{equation}

On the other hand, $B_{\alpha}'=B_{LR_{\alpha}^{-1}}(x_{\alpha})$, where $L>0$ is so large that
\begin{equation}\label{a54.6}
  \int\limits_{B_{L}(0)} |U_j|^2dx>0,\ j=1,\cdots,k.
\end{equation}
Since $u_{\alpha} = 0$ in $\mathbb{R}^N \setminus\Omega$, we have
\begin{align}\label{a54}
  \int\limits_{B_{\alpha}}
  |u_{\alpha}|^{2}dx &= \int\limits_{B_{t_{\alpha}R^{-\frac{1}{2}}_{\alpha}}(x_{\alpha})}|u_{\alpha}|^{2}dx\geq  \int\limits_{B'_{\alpha}}
  |u_{\alpha}|^{2}dx\nonumber\\
  &\geq \frac{1}{2}\int\limits_{B'_{\alpha}}|u_{\alpha,1}|^{2}dx
  -C\int\limits_{B'_{\alpha}}|u_{0}|^{2}dx-C\int\limits_{B'_{\alpha}}|u_{\alpha,2}|^{2}dx.
\end{align}
But
\begin{equation}\label{a55}
 C\int\limits_{B'_{\alpha}}|u|^{2}dx\leq CR_{\alpha}^{-N}=o(1)R_{\alpha}^{-2},
\end{equation}
and
\begin{equation}\label{a56}
  \int\limits_{B'_{\alpha}}|u_{\alpha,2}|^{2}dx\leq C\left(\int\limits_{B'_{\alpha}}|u_{\alpha,2}|^{2^{\ast}}dx\right)^{\frac{2}{2^{\ast}}}
 )R_{\alpha}^{-2}=o(1)R_{\alpha}^{-2}.
\end{equation}
Since $\|u_{\alpha,2}\|\rightarrow 0$ as $\alpha\rightarrow1$.

On the other hand, let us assume that $\rho_{x_{\alpha,1},R_{\alpha,1}}(U_1)$ is the bubble with slowest concentration
rate. Then
\begin{equation}\label{a57}
  \int\limits_{B'_{\alpha}}|u_{\alpha,1}|^{2}dx\geq \frac{1}{2} \int\limits_{B'_{\alpha}}|\rho_{x_{\alpha,1},R_{\alpha,1}}(U_1)|^2dx
  +O\left(\sum_{j=2}^k\int\limits_{B'_{\alpha}}|\rho_{x_{\alpha,j},R_{\alpha,j}}(U_j)|^2dx\right).
\end{equation}
Direct calculations implies that
$$ \int\limits_{B'_{\alpha}}|\rho_{x_{\alpha,1},R_{\alpha,1}}(U_1)|^2dx=
 R_{\alpha,1}^{-2}\int\limits_{B_{L}(0)}|U_1|^2dx\geq C_1R_{\alpha,1}^{-2},$$
for some constant $C_1>0$. Similarly,
\begin{equation}\label{a58}
\int\limits_{B'_{\alpha}}|\rho_{x_{\alpha,j},R_{\alpha,j}}(U_j)|^2dx=
 R_{\alpha,j}^{-2}\int\limits_{(B'_{\alpha})_{x_{\alpha,j},R_{\alpha,j}}}|U_j|^2dx.
\end{equation}
Here we use the notation $S_{x,R} = \{y: R^{-1}y + x \in  S\}$ for any set S.

If $\frac{R_{\alpha,j}}{R_{\alpha,1}}\rightarrow +\infty$, then we obtain from (\ref{a58}).
\begin{equation}\label{a59}
\int\limits_{B'_{\alpha}}|\rho_{x_{\alpha,j},R_{\alpha,j}}(U_j)|^2dx=
o\left(R_{\alpha,1}^{-2}\right).
\end{equation}

If $\frac{R_{\alpha,j}}{R_{\alpha,1}}\leq C<+\infty$, then

\begin{align}\label{a60}
  (B'_{\alpha})_{x_{\alpha,j},R_{\alpha,j}} &= \left\{y: R_{\alpha,j}^{-1}y + x_{\alpha,j} \in  B'_{\alpha}\right\}\nonumber\\
  & =\left\{y: \left|R_{\alpha,j}^{-1}y + x_{\alpha,j}- x_{\alpha,1}\right|\leq LR_{\alpha,1}^{-1} \right\} \subset
   \left\{y: \left|y + R_{\alpha,j}(x_{\alpha,j}- x_{\alpha,1})\right|\leq C\right\}.
\end{align}

Since $\left|R_{\alpha,j}(x_{\alpha,j}- x_{\alpha,1})\right|$ as $\alpha\rightarrow 1$,
we find that $(B'_{\alpha})_{x_{\alpha,j},R_{\alpha,j}} $ moves to infinity. So we
obtain from (\ref{a54.5} and (\ref{a58}) that
\begin{equation}\label{a61}
\int\limits_{B'_{\alpha}}|\rho_{x_{\alpha,j},R_{\alpha,j}}(U_j)|^2dx=
o\left(R_{\alpha,1}^{-2}\right).
\end{equation}
So, we have proved that there is a constant $C_1> 0$, such that
\begin{equation}\label{a62}
  \int\limits_{B'_{\alpha}}|u_{\alpha,1}|^{2}dx\geq C_1R_{\alpha}^{-2}.
\end{equation}
Therefore, by \eqref{a54},\eqref{a55},\eqref{a56} and \eqref{a62}, we have proved
\begin{align}\label{a63}
  LHS of (\ref{a51}) \geq \frac{C_1}{4}R^{-2}_{\alpha}.
\end{align}
From (\ref{a52}) and (\ref{a63}), we find

\begin{align}\label{a64}
 R^{-2}_{\alpha} \leq CR^{-\frac{N-2}{2}}_{\alpha}.
\end{align}
which is a contradiction if $N >6$.
\end{proof}

\begin{proof}[\textbf{Proof of Theorem \ref{rrr}}]
For any $k \in  \mathbb{N}$, define the $Z_2$-homotopy class $\mathcal{F}_k$ by
$$\mathcal{F}_k =\{A:  A \in H^{1,2\alpha}_0 (\Omega) \ \mbox{is} \ \mbox{compact},\ \mathbb{Z}_2-\mbox{invariant}, \mbox{and} \ \gamma (A) \geq k\},$$
where the genus $\gamma(A)$ is smallest integer $m$, such that there exists an odd map $\phi\in C(A, \mathbb{R}^m\setminus
\{0\})$.

For $k=1,2,\cdots,$ we can define the minimax value (see\cite[p.134]{gho} )

$$c_{k,\alpha}=\inf_{A\in \mathcal{F}_k} \max_{u\in A} I_{\alpha,\lambda}(u).$$

From Corollary 7.12 in \cite{gho}, for each small $\alpha >0$, $c_{k,\alpha}$ is a critical value of $I_{\alpha,\lambda}(u)$,
since $I_{\alpha,\lambda}(u)$
satisfies the Palais-Smale condition. Thus (1.3) has a solution $u_{k,\alpha}$  such that $I_{\alpha,\lambda}(u)=c_{k,\alpha}$.

 We fix a $\alpha_0<\frac{N}{N-2}$ and $\alpha<\frac{\alpha_0}{2}$.  Since $c_{k,\alpha}$  is is increasing in $\alpha>1$,
 we  obtain $c_{k,\alpha} \leq c_{k,\alpha_0}$.
So $c_{k,\alpha}$ is uniformly bounded for fixed $k$.

 From $I_{\alpha,\lambda}(u_{k,\alpha})= c_{k,\alpha}$ and
the equation satisfied by $c_{k,\alpha_0}$, we get
\begin{align}\label{a65}
\int\limits_{\Omega}u_{k,\alpha}^{2^{\ast}}dx&=\int\limits_{\Omega} (1+|\nabla u_{k,\alpha}|^2)^{\alpha-1}|\nabla u_{k,\alpha} |^2dx
-\lambda\int\limits_{\Omega}u_{k,\alpha}^2dx
\end{align}
and
\begin{align}\label{a66}
\frac{1}{2^{\ast}}\int\limits_{\Omega}|u_{k,\alpha}|^{2^{\ast}}dx+c_{k,\alpha}&=
\frac{1}{2\alpha}\int\limits_{\Omega} [(1+|\nabla u_{k,\alpha}|^2)^{\alpha}-1 ]dx
-\frac{\lambda}{2}\int\limits_{\Omega}u_{k,\alpha}^2dx.
\end{align}
So from (\ref{a65}) and (\ref{a66})  we have
\begin{align}\label{a67}
c_{k,0}> c_{k,\alpha}&=\left(\frac{1}{2\alpha}-\frac{1}{2^{\ast}}\right)\int\limits_{\Omega}|u_{k,\alpha}|^{2^{\ast}}dx
-\left(\frac{\lambda}{2}-\frac{\lambda}{2\alpha}\right)\int\limits_{\Omega}|u_{k,\alpha}|^{2}dx\nonumber\\
&\geq \left(\frac{1}{2\alpha}-\frac{1}{2^{\ast}}\right)\int\limits_{\Omega}|u_{k,\alpha}|^{2^{\ast}}dx
- o(|\alpha-1|)C\lambda\left(\int\limits_{\Omega}|u_{k,\alpha}|^{2^{\ast}}dx\right)^{\frac{2}{2^{\ast}}},
\end{align}
where $C$ depends on $|\Omega|$, $2\alpha$ and $N$ only.

Therefore there is a positive constant $C$ independent
of $N$ such that
$$\int\limits_{\Omega} |\nabla u_{k,\alpha}|^{2}dx\leq C
\int\limits_{\Omega} |\nabla u_{k,\alpha}|^{2\alpha}dx\leq C.$$

Thus 	$u_{k,\alpha}$	 is uniformly bounded with respect to $\alpha$. And the bubble $\rho_{x_{\alpha,j},R_{\alpha,j}}(U_j)$
does not appear in (\ref{a53}). So we have a subsequence of $\{u_{k,\alpha}\}$, such that, $u_{k,\alpha}\rightarrow u_k$ in $H^1_0(\Omega)$,
 and $c_{k,\alpha}\rightarrow c_k$ as $\alpha\rightarrow1$. Then $u_k$ is a critical point of $I_{1,\lambda}(u)$ and $I_{1,\lambda}(u_k)=c_k$.

We are now ready to show that $I_{1,\lambda}(u)$ has infinitely many critical points. Noting that $c_k$ is
non-decreasing in $k$, we have the following two cases:
\begin{itemize}\addtolength{\itemsep}{-1.5 em} \setlength{\itemsep}{-5pt}
\item[\textbf{Case I}.]There are $1 < k_1 < \cdots < k_i < \cdots$, satisfying $c_{k_1} < \cdots < c_{k_i} < \cdots$.

In this case,  $I_{1,\lambda}(u)$  has
infinitely many critical points  $u_i$ such that $I (u_i )= c_{k_i}$ .
\item[\textbf{Case II}.]There is a positive integer $m$ such that $c_k = c$ for all $k \geq  m$.
\end{itemize}
If for any $\delta > 0$, $I_{1,\lambda}(u)$  has a critical point $u$ with $I_{1,\lambda}(u) \in  (c -\delta, c +\delta)$
and $I_{1,\lambda}(u)\not=c$, then we
are done. So from now on we assume that there exists a $\delta >0$, such that $I_{1,\lambda}(u)$  has no critical point
$u$ with $I_{1,\lambda}(u) \in (c - \delta, c) \cup (c, c +\delta)$.

 Let
 $$K_c = \{u \in H^{1,2}_0(\Omega): I'_{1,\lambda}(u)= 0, I_{1,\lambda}(u)= c\}.$$
If   we can  prove  that
$$\gamma(K_c)\geq 2,$$ then $I_{1,\lambda}(u)$ has infinitely many
critical points.

Suppose, on the contrary, that $\gamma(K_c)= 1$. Take a small $\delta_1 > 0$, such that $\gamma(\mathcal{K})= 1$, where
$\mathcal{K}=\{u\in H_0^{1,2}(\Omega): \|u-K_{c}\|\leq \delta_1\}.$  Then $\gamma(\mathcal{K}\cap H^{1,2\alpha}_0(\Omega))\leq  1$.

Define
$$D_{\alpha}=\left(K_{\alpha}^{c+\delta}\setminus K_{\alpha}^{c-\delta}\right)\setminus (\mathcal{K}\cap H^{1,2\alpha}_0(\Omega)),$$
where
$$K_{\alpha}^{t}=\{u \in H^{1,2\alpha}_0(\Omega): I_{\alpha,\lambda}(u)\leq t\}.$$
We now claim that if $ \alpha -1> 0 $ is small, $I_{\alpha,\lambda}(u)$ has no critical point $u \in  D_{\alpha}$. Otherwise, suppose
that there are $\alpha\rightarrow 1$ and $u_n\in D_{\alpha}$ satisfying
$$I'_{\alpha,\lambda}(u_{\alpha})=0,\ u_{\alpha}\not\in \mathcal{K}\cap H^{1,2\alpha}_0(\Omega).$$
And the bubble $\rho_{x_{\alpha,j},R_{\alpha,j}}(U_j)$
does not appear in (\ref{a53}). So $u_{\alpha}$ (up to a subsequence) converges strongly to $u$ in $H_0^{1}(\Omega)$
as $\alpha\rightarrow 1$.  Then
$$I'_{1,\lambda}(u)=0,\  I_{1,\lambda}(u)\in (c-\delta, c+\delta),\ \ u\not\in\mathcal{K}.$$
This contradicts to the assumption.

So, for any $\varepsilon >0$ small, there exists a constant $c^{\ast}_{\alpha} > 0$, such that
$$\|I'_{1,\lambda}(u)\|\geq c^{\ast}_{\alpha} >0,\mbox{for\ all}\ u\in D_{\alpha}.$$
Standard techniques show that we can find an odd homeomorphism $\eta: H_0^{1,2\alpha}(\Omega)\rightarrow H_0^{1,2\alpha}(\Omega)$
such that
\begin{equation}\label{a68}
\eta(\left(K_{\alpha}^{c+\delta}\setminus  (\mathcal{K}\cap H^{1,2\alpha}_0(\Omega))\right)\subset K_{\alpha}^{c-\delta}.
\end{equation}
See for example the proof of Theorem 1.9 in \cite{rab}.


Fix $k>m$, Since $c_{k,\alpha}, c_{k+1,\alpha} \rightarrow c$ as $\alpha\rightarrow1$, we can find an $\alpha-1 >0$ small, such that
$$c_{k,\alpha}, c_{k+1,\alpha} \in \left(c-\frac{1}{4}\delta, c+\frac{1}{4}\delta\right).$$
By the definition of $c_{k+1,\alpha} $, we can find a set $A \in \mathcal{F}_{k+1}$, such that
$$I_{\alpha,\lambda}(u)<c_{k+1,\alpha}+\frac{1}{4}\delta<c+\delta,\ \ u\in A.$$
Hence, $A\subset K_{\alpha}^{c+\delta}$. From \eqref{a68}, $\tilde{A}=:\eta(A\setminus
(\mathcal{K}\cap H^{1,2\alpha}_0(\Omega)) )\subset K_{\alpha}^{c-\delta}$. That is
$$I_{\alpha,\lambda}(u)<c-\delta,\ \ u\in \tilde{A}.$$

On the other hand, if
$\gamma((\mathcal{K}\cap H^{1,2\alpha}_0(\Omega)))=0$, then
$\eta\left(K_{\alpha}^{c+\delta}\right)\subset K_{\alpha}^{c-\delta}$, which is contradiction with the definition of
$c_{k,\alpha}$.

If $\gamma((\mathcal{K}\cap H^{1,2\alpha}_0(\Omega)))=1$, by Lemma 3.32 of \cite{rab}, we find that
$A \setminus (\mathcal{K}\cap H^{1,2\alpha}_0(\Omega)) \subset \mathcal{F}_k$ . Using Theorem 1.9 in \cite{rab},
 we conclude$\tilde{A}\subset \mathcal{F}_k$ . As a result,

$$c_{k,\alpha}\leq \sup_{u\in \tilde{A}} I_{\alpha,\lambda}(u)<c-\delta.$$
This is a contradiction to $c_{k,\alpha} > c -\frac{1}{4}\delta$.

\end{proof}

 \appendix
\numberwithin{equation}{theorem}
  \section*{Appendix. Some estimates for solutions}
   \renewcommand\thetheorem{ A }
   In this section, we assume that $\Omega_1$ is a bounded domain in $\mathbb{R}^N$. We give some estimates for solutions of the equation $-\mbox{div}((1+|\nabla \cdot|^2)^{\alpha-1}\nabla \cdot)=f$.    These estimates  are very similar to the estimates of $p$-Laplacian  equation which were obtained in
 (see\cite{cao}). But   for the readers’ convenience,  we give the details of  these estimates.
\begin{proposition}\label{app1}
Let $ w \in H_0^{1,2\alpha} (\Omega_1) $ be the solution of
\begin{equation} \label{app2}
-\emph{div}((1+|\nabla w|^2)^{\alpha-1}\nabla w)=f(x)\ \ \ \mbox{in}\ \Omega_1.
\end{equation}
Suppose that $f \geq 0, f \in  L^{\infty}(\Omega_1)$. Then, for any
$\frac{N}{2\alpha}>q\geq 1$, there is a constant $C = C(q)$, such
that
$$\|w\|_{\frac{Nq(2\alpha-1)}{N-2\alpha q}}\leq C\|f\|_q^{\frac{1}{2\alpha-1}}.$$
\end{proposition}

\begin{proof}
We now prove that for $r>1-\frac{1}{2\alpha}$,
\begin{equation}\label{app5}
 \int\limits_{\Omega_1}(1+|\nabla w|^2)^{\alpha-1}\nabla w \nabla (w^{1+2\alpha(r-1)})dx\leq \int\limits_{\Omega_1}f(x)w^{1+2\alpha(r-1)}dx.
\end{equation}
Firstly, we assume $r\geq 1$ and $\eta=w^{1+2\alpha(r-1)}$. Then from
$$\nabla \eta=(1+2\alpha(r-1))w^{2\alpha(r-1)}\nabla w,$$
and $r\geq 1$,  it is easy to show that $\eta\in H_0^{1,2\alpha} (\Omega_1)$ since $w\in L^{\infty}(\Omega_1)$. So, we obtain
$$\int\limits_{\Omega_1}(1+|\nabla w|^2)^{\alpha-1}\nabla w \nabla (w^{1+2\alpha(r-1)})dx= \int\limits_{\Omega_1}f(x)w^{1+2\alpha(r-1)}dx.$$
Hence Proposition \ref{app1} holds.

Next we consider the case $r\in \left(1-\frac{1}{2\alpha},1\right)$. In this case, $w^{1+2\alpha(r-1)}$ may not be in $H_0^{1,2\alpha} (\Omega_1)$.
So we need to proceed differently. By the comparison principle, we know that $w\geq 0$ in $\Omega_1$. For any $\theta>0$ being a small number,
let $\eta=w(w+\theta)^{2\alpha(r-1)}$. Then $\eta\in H_0^{1,2\alpha} (\Omega_1)$, and
$$\nabla\eta=(w+\theta)^{2\alpha(r-1)}\nabla w+2\alpha(r-1)(w+\theta)^{2\alpha(r-1)-1}w\nabla w.$$
So we deduce
\begin{align}\label{app4}
& \int\limits_{\Omega_1}(1+|\nabla w|^2)^{\alpha-1}\nabla w \nabla \eta)dx\nonumber\\
&=\int\limits_{\Omega_1}(1+|\nabla w|^2)^{\alpha-1}|\nabla w|^2 ((w+\theta)^{2\alpha(r-1)}+2\alpha(r-1)(w+\theta)^{2\alpha(r-1)-1}w)dx\nonumber\\
&=\int\limits_{\Omega_1}f(x)w(w+\theta)^{2\alpha(r-1)}dx.
\end{align}

On the other hand, by Fatou's lemma,
\begin{align}\label{app6}
& (1+2\alpha(r-1))\int\limits_{\Omega_1} w^{2\alpha(r-1)}(1+|\nabla w|^2)^{\alpha-1}|\nabla w|^2dx \nonumber\\
& =\int\limits_{\Omega_1}\liminf_{\theta\rightarrow 0}(1+|\nabla w|^2)^{\alpha-1}\nabla w \nabla \eta)dx\nonumber\\
&=\int\limits_{\Omega_1}\liminf_{\theta\rightarrow 0}(1+|\nabla w|^2)^{\alpha-1}|\nabla w|^2 ((w+\theta)^{2\alpha(r-1)}+2\alpha(r-1)(w+\theta)^{2\alpha(r-1)-1}w)dx\nonumber\\
&\leq \liminf_{\theta\rightarrow 0}\int\limits_{\Omega_1}(1+|\nabla w|^2)^{\alpha-1}|\nabla w|^2 ((w+\theta)^{2\alpha(r-1)}+2\alpha(r-1)(w+\theta)^{2\alpha(r-1)-1}w)dx\nonumber\\
&=\liminf_{\theta\rightarrow 0}\int\limits_{\Omega_1}f(x)w(w+\theta)^{2\alpha(r-1)}dx dx\nonumber\\
&\leq \int\limits_{\Omega_1}f(x)w^{1+2\alpha(r-1)}dx.
\end{align}
Hence, \eqref{app5} holds.

From (\ref{app5}), by Sobolev embedding,  H\"{o}lder inequality and $\nabla w^r=rw^{r-1}\nabla w$,  we have
\begin{equation}\label{app7}
C\|w\|_{(2\alpha)^{\ast}r}^{2\alpha r}\leq \int\limits_{\Omega_1} f(x)w^{1+2\alpha(r-1)}dx\leq \|f\|_q\left(w^{1+2\alpha(r-1)\frac{q}{q-1}}\right)^{1-\frac{1}{q}},
\end{equation}
where $(2\alpha)^{\ast}=\frac{2N\alpha}{N-2\alpha}$.
Choose $r$, such that
$$(1+2\alpha(r-1))\frac{q}{q-1}=(2\alpha)^{\ast}r.$$
Then,
\begin{equation}\label{app9}
 r=\frac{q(2\alpha-1)}{2q\alpha-(2\alpha)^{\ast}(q-1)}.
\end{equation}
Note that $2q\alpha-(2\alpha)^{\ast}(q-1)>0$ since $q<\frac{N}{2\alpha}$. So, $r>0$ and
$$r>\frac{q(2\alpha-1)}{2q\alpha}=1-\frac{1}{2\alpha}.$$
For such $r$,
$$(2\alpha)^{\ast}r=\frac{(2\alpha)^{\ast}q(2\alpha-1)}{2\alpha q-(2\alpha)^{\ast}(q-1)}=\frac{Nq(2\alpha-1)}{N-2\alpha q}.$$
Thus, \eqref{app7} implies that
\begin{equation}\label{app8}
C\|w\|_{(2\alpha)^{\ast}r}^{2\alpha r-(2\alpha)^{\ast}r\left(1-\frac{1}{q}\right)}\leq \|f\|_q.
\end{equation}
Moreover, by \eqref{app9}, we know that
$$2\alpha r-(2\alpha)^{\ast}r\left(1-\frac{1}{q}\right)=2\alpha-1.$$
So, from the above discussion,  Proposition \ref{app1} holds.

\end{proof}

\begin{corollary}\label{app3}
Let $ w \in H_0^{1,2\alpha} (\Omega_1) $ be the solution of
\begin{equation} \label{app10}
-\emph{div}((1+|\nabla w|^2)^{\alpha-1}\nabla w)=a(x)v^{2\alpha-1}\ \ \ \mbox{in}\ \Omega_1,
\end{equation}
where   $a(x) \geq 0, v \geq 0$ are functions satisfying $a,v\in L^{\infty}(\Omega_1)$. Then, for any
$q>\left(1-\frac{1}{2\alpha}\right)\frac{2\alpha}{N-2\alpha}$, there is a constant $C = C(q)$, such
that
$$\|w\|_{q}\leq C\|a\|_{\frac{N}{2\alpha}}^{\frac{1}{2\alpha-1}}\|v\|_q.$$
\end{corollary}
 \begin{proof}
  For any $q_1\in \left(1,\frac{N}{2\alpha}\right)$, set $q=\frac{Nq_1(2\alpha-1)}{N-2\alpha q_1}$. We easily prove that
  $\frac{2N}{2\alpha}>q_1>1$ is equivalent to $\left(1-\frac{1}{2\alpha}\right)\frac{2\alpha}{N-2\alpha}<q<+\infty.$
  By Proposition\ref{app1}, there exists a constant $C(q_1) > 0$ such that
 \begin{equation}\label{app11}
   \|w\|_{q}\leq C(q_1)\|av^{2\alpha-1}\|_{\frac{N}{2\alpha}}^{\frac{1}{2\alpha-1}}.
 \end{equation}
On the other hand, from H\"{o}lder inequality we deduce
  \begin{equation}\label{app12}
\|av^{2\alpha-1}\|_{q_1}^{\frac{1}{2\alpha-1}}\leq \|a\|_{\frac{N}{2\alpha}}^{\frac{1}{2\alpha-1}}\|v\|_q.
  \end{equation}
  which, together with (\ref{app11}), completes our proof of Corollary \ref{app3}.
  \end{proof}
 \renewcommand\thetheorem{B }
 \begin{corollary}\label{app13}
Let $ w \in H_0^{1,2\alpha} (\Omega_1) $ be the solution of
\begin{equation} \label{app14}
-\emph{div}((1+|\nabla w|^2)^{\alpha-1}\nabla w)=a(x)v^{2\alpha-1}\ \ \ \mbox{in}\ \Omega_1,
\end{equation}
where   $a(x) \geq 0, v \geq 0$ are functions satisfying $a,v\in L^{\infty}(\Omega_1)$. Then, for any
$p_2\in \left(\frac{(2\alpha-1)N}{N-2\alpha}, (2\alpha)^{\ast}\right)$, there is a constant $C = C(p_2)$, such
that
$$\|w\|_{p_2}\leq C\|a\|_{r}^{\frac{1}{2\alpha-1}}\|v\|_{(2\alpha)^{2^\ast}},$$
where $r$ is determined by $\frac{1}{r}=\frac{2\alpha-1}{p_2}+\frac{2\alpha}{N}-\frac{2\alpha-1}{(2\alpha)^{\ast}}$.
\end{corollary}
 \begin{proof}
By the assumption we know that $w\geq 0$. Set $p_2=\frac{Nq_2(2\alpha-1)}{N-2\alpha q_2}$ and choose $r>0$ such that
\begin{equation}\label{app15}
  \frac{1}{r}=\frac{2\alpha-1}{p_2}+\frac{2\alpha}{N}-\frac{2\alpha-1}{(2\alpha)^{2^{\ast}}},
\end{equation}
then $\frac{(2\alpha-1)rq_2}{r-q_2}=(2\alpha)^{\ast}$. It is easy to show  that $\frac{(2\alpha-1)N}{N-2\alpha}<p_2<(2\alpha)^{\ast}$
is equivalent to
$$1<q_2<\frac{N}{2\alpha+(N-2\alpha)\left(1-\frac{1}{2\alpha}\right)},$$
which implies that $1 < q_2 < \frac{N}{2\alpha}$. For such $p_2$ and $q_2 $ we can apply Proposition \ref{app1} to obtain a
constant $C(q_2) > 0$ such that
$$\|w\|_{p_2}\leq C(q_2)\|a v^{2\alpha-1}\|_{q_2}^{\frac{1}{2\alpha-1}},$$
On the other hand, from  H\"{o}lder inequality we get
  \begin{equation}\label{app16}
\|av^{2\alpha-1}\|_{q_2}^{\frac{1}{2\alpha-1}}\leq \|a\|_{r}^{\frac{1}{2\alpha-1}}\|v\|_{(2\alpha)^{\ast}}.
  \end{equation}
Hence, Corollary \ref{app13} holds.
 \end{proof}

\section*{Acknowledgments}
F. Fang was supported by  the funds from NSFC (No. 11626038).

\end{document}